\documentclass[11pt]{article}

\usepackage{amssymb,amsmath,amsfonts,amsthm}
\usepackage{latexsym}
\usepackage{graphics}
\usepackage{indentfirst}

\newtheorem*{thm*}{Theorem}
\newtheorem*{lem*}{Lemma}
\newtheorem*{pro*}{Proposition}
\newtheorem{thm}{Theorem}
\newtheorem{lem}[thm]{Lemma}
\newtheorem{pro}[thm]{Proposition}

\newtheorem{cor}[thm]{Corollary}
\newtheorem{conj}[thm]{Conjecture}

\newcommand{\N}{\mathbb{N}}
\newcommand{\Z}{\mathbb{Z}}

\newcommand{\col}{\mathrm{col}}

\begin{document}

\title{Criticality, The List Color Function, and List Coloring the Cartesian Product of Graphs}

\author{Hemanshu Kaul\footnote{Department of Applied Mathematics, Illinois Institute of Technology, Chicago, IL 60616. E-mail: {\tt kaul@iit.edu}} \;and Jeffrey A. Mudrock \footnote{Department of Applied Mathematics, Illinois Institute of Technology, Chicago, IL 60616. E-mail: {\tt jmudrock@hawk.iit.edu}} }

\date{}

\maketitle

\begin{abstract}
We introduce a notion of color-criticality in the context of chromatic-choosability, $\chi_{\ell}(G)=\chi(G)$. We define a graph $G$ to be strong $k$-chromatic-choosable if $\chi(G) = k$ and every $(k-1)$-assignment for which $G$ is not list-colorable has the property that the lists are the same for all vertices. That is the usual coloring is, in some sense, the obstacle to list-coloring. We prove basic properties of strongly chromatic-choosable graphs such as chromatic-choosability and vertex-criticality, and we construct infinite families of strongly chromatic-choosable graphs. We derive a sufficient condition for the existence of at least two list colorings of strongly chromatic-choosable graphs and use it to show that: if $M$ is a strong $k$-chromatic-choosable graph with $|E(M)| \leq |V(M)|(k-2)$ and $H$ is a graph that contains a Hamilton path, $w_1, w_2, \ldots, w_m$, such that $w_i$ has at most $\rho \geq 1$ neighbors among $w_1, \ldots, w_{i-1}$, then $\chi_{\ell}(M \square H) \le k+ \rho - 1$. We show that this bound is sharp for all $\rho \ge 1$ by generalizing the theorem to apply to $H$ that are $(M,\rho)$-Cartesian accommodating which is a notion we define with the help of the list color function, $ P_{\ell}(G,k)$, the list analogue of the chromatic polynomial.

We use the list color function to determine the list chromatic number of certain star-like graphs: $\chi_{\ell}(M \square K_{1,s}) =$ $k \; \text{if } s < P_{\ell}(M,k)$, or $k+1 \; \text{if } s \geq P_{\ell}(M,k)$, where $M$ is a strong $k$-chromatic-choosable graph. We use the fact that $ P_{\ell}(M,k)$ equals $P(M,k)$, the chromatic polynomial, when $M$ is an odd cycle, complete graph, or the join of an odd cycle with a complete graph to prove $\chi_{\ell}(C_{2l+1} \square K_{1,s})$ transitions from 3 to 4 at $s= 2^{2l+1}-2$, $\chi_{\ell}(K_n \square K_{1,s})$ transitions from $n$ to $n+1$ at $s= n!$, and $\chi_{\ell}(K_n \vee C_{2l+1} \square K_{1,s})$ transitions from $n+3$ to $n+4$ at $s=\frac{1}{3} (n+3)! (4^l-1)$.



\end{abstract}

\section{Introduction}\label{intro}

In this paper all graphs are finite simple graphs.  Generally speaking we follow West~\cite{W01} for terminology and notation. List coloring is a well known variation on the classic vertex coloring problem, and it was introduced independently by Vizing~\cite{V76} and Erd\H{o}s, Rubin, and Taylor~\cite{ET79} in the 1970's.  In the classic vertex coloring problem we wish to color the vertices of a graph $G$ with as few colors as possible so that adjacent vertices receive different colors, a so-called \emph{proper coloring}. The chromatic number of a graph, denoted $\chi(G)$, is the smallest $k$ such that $G$ has a proper coloring that uses $k$ colors.  For list coloring, we associate a \emph{list assignment}, $L$, with a graph $G$ such that each vertex $v \in V(G)$ is assigned a list of colors $L(v)$.  A list assignment $L$ is called a \emph{$k$-assignment} if all the lists associated with $L$ have size $k$.  The graph $G$ is \emph{$L$-colorable} if there exists a proper coloring $f$ of $G$ such that $f(v) \in L(v)$ for each $v \in V(G)$ (we refer to $f$ as a \emph{proper $L$-coloring} for $G$). The \emph{list chromatic number} of a graph $G$, denoted $\chi_\ell(G)$, is the smallest $k$ such that $G$ is $L$-colorable whenever the list assignment $L$ satisfies $|L(v)| \geq k$ for each $v \in V(G)$. It is immediately obvious that for any graph $G$, $\chi(G) \leq \chi_\ell(G)$.  Erd\H{o}s, Taylor, and Rubin observed in~\cite{ET79} that bipartite graphs can have arbitrarily large list chromatic number. This means that the gap between $\chi(G)$ and $\chi_\ell(G)$ can be arbitrarily large.

\subsection{Chromatic-Choosability}\label{cc}

Graphs in which $\chi(G) = \chi_\ell(G)$ are known as \emph{chromatic-choosable} graphs~\cite{O02}. Many classes of graphs have been conjectured to be chromatic-choosable. The most well known conjecture along these lines is the List Coloring Conjecture which states that every line graph of a loopless multigraph is chromatic-choosable. The list coloring conjecture was formulated independently by many different researchers (see~\cite{HC92}). In addition, total graphs~\cite{BKW97} and claw free graphs~\cite{GM97} are conjectured to be chromatic-choosable. In 2001 Kostochka and Woodall~\cite{KW01} conjectured that the square of any graph is chromatic-choosable. However, Kim and Park proved this conjecture to be false~\cite{KP15}. On the other hand, there are classes of graphs that are known to be chromatic-choosable. In 1995, Galvin~\cite{G95} showed that the List Coloring Conjecture holds for line graphs of bipartite multigraphs, and in 1996, Kahn~\cite{K96} proved an asymptotic version of the conjecture. Tuza and Voigt~\cite{TV96} showed that chordal graphs are chromatic-choosable, and Prowse and Woodall~\cite{PW03} showed that powers of cycles are chromatic-choosable.  Recently, Noel, Reed, and Wu~\cite{NR15} proved Ohba's conjecture which states that every graph, $G$, on at most $2 \chi(G) + 1$ vertices is chromatic-choosable.

In this paper, we continue this investigation of chromatic-choosability in the realm of Cartesian products of graphs. To aid our investigations we use two major tools - criticality and the list color function.

\subsection{Critical Graphs}\label{cg}

Criticality is a widely used notion in the study of most graph properties. It is used to define the graphs where the property is lost by removal of an edge or vertex. Most commonly used in coloring problems, a \emph{$k$-critical} graph is a graph whose chromatic number is $k$ but whose proper subgraphs have chromatic number strictly less than $k$. If the proper subgraphs with this property are restricted to be just the induced subgraphs then such a critical graph is called \emph{$k$-vertex-critical}. In 1951 Dirac~\cite{D51} initiated the study of critical graphs and since then a large body of literature has developed around this notion since they characterize the chromatic number as: $\chi(G) \ge k$ if and only if $G$ contains a $k$-critical graph.

A similar notion for list coloring has been harder to study since list coloring is highly dependent on the particular list assignment. In 2009, Stiebitz, Tuza and Voigt~\cite{ST09} introduced and studied the notion of $k$-list critical. A graph $G$ is $L$-critical, for a list assignment $L$, if every proper subgraph of $G$ is $L$-colorable, but $G$ is not $L$-colorable. Then, $G$ is called $k$-list critical if there is a $(k-1)$-assignment $L$ for $G$ such that $G$ is $L$-critical. They studied the structure of such graphs, in particular list critical complete graphs. They introduced and studied an important subclass of list-critical graphs called \emph{strong k-critical} in which $G$ is $k$-critical and every $(k-1)$-assignment, $L$, for which $G$ is not $L$-colorable has the property that the lists are the same on all vertices. That is the usual coloring is, in some sense, the obstacle to list-coloring.

We extend this notion to define criticality in the context of chromatic-choosability. We define a graph $G$ to be \emph{strong k-chromatic-choosable} if $\chi(G) = k$ and every $(k-1)$-assignment, $L$, for which $G$ is not $L$-colorable has the property that the lists are the same on all vertices. These graphs include the strong $k$-critical graphs.  Furthermore, these graphs are chromatic-choosable and $k$-vertex-critical. In this paper we show how this notion can be used to study list coloring of Cartesian products of graphs.

\subsection{List Color Function}\label{lcf}

Counting the number of colorings of a graph is an important question that is studied systematically using the classical notion of the chromatic polynomial, originally introduced by Birkhoff in 1912~\cite{B12}. The \emph{chromatic polynomial} of a graph $G$ is the function $P(G,k)$ that is equal to the number of ordinary $k$-colorings of $G$.  It can be shown that $P(G,k)$ is a polynomial in $k$.  This notion was extended to list coloring as follows. If $L$ is a list assignment for $G$, we use $P(G,L)$ to denote the number of proper $L$-colorings of $G$. The \emph{list color function} $P_\ell(G,k)$ is the minimum value of $P(G,L)$ where the minimum is taken over all possible $k$-assignments $L$ for $G$. Generally, $P(G,k)$ and $P_{\ell}(G,k)$ can be quite different, especially for small values of $k$. However, as was reported in~\cite{AS90}, if $G$ is a chordal graph (i.e. a graph in which cycles of length at least four contain a chord), then $P_{\ell}(G,k)=P(G,k)$ for all $k$.  Also, Thomassen~\cite{T09} showed that if $G$ is a graph with $n$ vertices, then $P_{\ell}(G,k)=P(G,k)$ whenever $k \geq n^{10}$. Recently Wang, Qian, and Yan~\cite{WQ17} gave a major improvement: if $G$ is a connected graph with $m$ edges, then $P_{\ell}(G,k)=P(G,k)$ whenever $k > \frac{m-1}{\ln(1+ \sqrt{2})}$.

In this paper, we show how knowledge of the list color function can be exploited to bound the the list chromatic number of certain Cartesian products or even find it exactly in some special cases.

\subsection{Cartesian Product of Graphs}\label{prelim}

The \emph{Cartesian product} of graphs $G$ and $H$, denoted $G \square H$, is the graph with vertex set $V(G) \times V(H)$ and edges created so that $(u,v)$ is adjacent to $(u',v')$ if and only if either $u=u'$ and $vv' \in E(H)$ or $v=v'$ and $uu' \in E(G)$. Note that $G \square H$ contains $|V(G)|$ copies of $H$ and $|V(H)|$ copies of $G$.  It is also easy to show that $\chi(G \square H) = \max \{\chi(G), \chi(H) \}$.  So, we have that $\max \{\chi(G), \chi(H) \} \leq \chi_\ell(G \square H)$.

There are few results in the literature regarding the list chromatic number of the Cartesian product of graphs. In 2006, Borowiecki, Jendrol, Kr{\'a}l, and Mi{\v s}kuf~\cite{BJ06} showed the following.
\begin{thm}[\cite{BJ06}] \label{thm: Borow}
$\chi_\ell(G \square H) \leq \min \{\chi_\ell(G) + \col(H), \col(G) + \chi_\ell(H) \} - 1.$
\end{thm}

Here $\col(G)$, the \emph{coloring number} of a graph $G$, is the smallest integer $d$ for which there exists an ordering, $v_1, v_2, \ldots, v_n$, of the elements in $V(G)$ such that each vertex $v_i$ has at most $d-1$ neighbors among $v_1, v_2, \ldots, v_{i-1}$. The coloring number is a classic greedy upper bound on the list chromatic number, and it immediately implies that $\Delta(G)+1$ is an upper bound on the list chromatic number where $\Delta(G)$ is the maximum degree of $G$. Vizing~\cite{V76} extended this by proving the list coloring version of Brooks' Theorem characterizing complete graphs and odd cycles as the only connected graphs with $\chi_\ell(G)=\Delta(G)+1$.


Borowiecki et al.~\cite{BJ06} construct examples where the upper bound in their theorem is tight.  Specifically, they show that if $k \in \N$ and $G$ is a copy of the complete bipartite graph $K_{k, (2k)^{k(k+k^k)}}$, then $\chi_\ell(G \square G)=\chi_\ell(G)+\col(G)-1$.  On the other hand, there are examples where the upper bounds from Theorem~\ref{thm: Borow} and Vizing's extension of Brooks' theorem are not tight. For example, suppose that $G$ is a copy of $C_{2k+1} \square P_m$ where $m \geq 3$.  Since $\chi_\ell(C_{2k+1})=\col(C_{2k+1})=3$ and $\chi_\ell(P_m)=\col(P_m)=2$, both the bounds give us $\chi_\ell(G) \leq 4$, yet we will show below that $\chi_\ell(G)=3$.  Another, more dramatic, example where these theorems do not produce tight bounds is when we are working with the Cartesian product of two complete graphs.  Suppose $m \geq n \geq 2$, and note that $K_m \square K_n$ is the line graph for the complete bipartite graph $K_{m,n}$. So, by Galvin's celebrated result~\cite{G95}: $m=\chi(K_m \square K_n) = \chi_\ell(K_m \square K_n)$. However, Theorem~\ref{thm: Borow} only yields an upper bound of $m+n-1$ since $\chi_\ell(K_m)=\col(K_m)=m$, and Brooks' theorem only tells us $\chi_\ell(K_m \square K_n) \leq m+n-2$.

Alon~\cite{A00} showed that for any graph $G$, $\col(G) \leq 2^{O(\chi_\ell(G))}$.  Combining this result with Theorem~\ref{thm: Borow} implies that we have an upper bound on $\chi_\ell(G \square H)$ in terms of only the list chromatic numbers of the factors.  Borowiecki et al.~\cite{BJ06} conjecture that a much stronger bound holds: there is a constant $A$ such that $\chi_\ell(G \square H) \leq A( \chi_\ell(G) + \chi_\ell(H))$.  While we will not address this conjecture in this paper, we will present results that are improvements on Theorem~\ref{thm: Borow} when the factors in the Cartesian product belong to certain large classes of graphs.

\subsection{Outline of the Paper and Open Questions}

The following open questions motivated much of the research presented in this paper.

\vspace{5mm}
\noindent {\bf Question:} For what factors, $G$ and $H$, is $G \square H$ chromatic-choosable?

Or, even more simply,

\vspace{5mm}
\noindent {\bf Question:} For which graphs, $G$, is $G \square P_n$ chromatic-choosable for each $n \in \N$?
\vspace{5mm}

In Section 2, we introduce the notion of \emph{strongly chromatic-choosable} graphs, and we prove some basic facts about these graphs including that they are chromatic-choosable, and $k$-vertex-critical in a strong sense: $\chi(G-\{v\}) \leq \chi_{\ell}(G - \{v\})<k$, for any $v \in V(G)$. This extends the notion of \emph{strongly critical} graphs. In fact, a graph is strong $k$-critical if and only if it is $k$-critical and strong $k$-chromatic-choosable. In addition to strongly critical graphs like complete graphs, odd cycles, Dirac graphs. etc., we also construct nontrivial examples of strong $k$-chromatic choosable graphs that are not strong k-critical, for all $k \ge 4$. We also discuss how strong chromatic-choosability is related to an older notion of amenable colorings and to a conjecture of Mohar from 2001 regarding $\Delta(G)+1$-edge critical graphs.

In Section 3, we derive a sufficient condition for existence of at least two colorings of strongly chromatic-choosable graphs using a result of Akbari, Mirrokni, and Sadjad~\cite{AM06} that relates $f$-choosability and unique list colorability.

\begin{lem*}
Let $G$ be a strong $k$-chromatic-choosable graph with $n$ vertices and $m$ edges. Suppose that $L$ is a list assignment for $G$ such that $|L(v)| \geq k-1$ for each $v \in V(G)$ and $L$ is not a constant $(k-1)$-assignment for $G$. If $m \leq n(k-2)$, then there are at least two proper $L$-colorings for $G$.
\end{lem*}

All strong $3$-chromatic-choosable graphs and all the examples of strong $k$-chromatic-choosable graphs with $k \geq 4$ from Section 2 satisfy this \emph{edge condition}: $m \leq n(k-2)$. We construct strong $k$-chromatic-choosable graphs that do not satisfy this condition for $k=4,5,6,7$. These examples lead us to question whether we really need the edge condition in the above lemma.

\vspace{5mm}
\noindent {\bf Question:} For $k \geq 4$, does there exist a graph $G$ that is strong $k$-chromatic-choosable and violates the edge condition, and a non-constant $(k-1)$-assignment, $L$, for $G$ such that $G$ has a unique proper $L$-coloring?
\vspace{5mm}

In Section 4, we use the existence of these multiple colorings in strongly chromatic-choosable graphs to prove the following.

\begin{thm*} Let $M$ be a strong $k$-chromatic-choosable graph with $n$ vertices and $m$ edges that satisfies $m \leq n(k-2)$, and suppose that graph $H$ contains a Hamilton path, $w_1, w_2, \ldots, w_m$, such that $w_i$ has at most $\rho \ge 1$ neighbors among $w_1, \ldots, w_{i-1}$. Then, $\chi_{\ell}(M \square H) \leq k+ \rho -1$.
\end{thm*}

This theorem is sharp for $\rho =1$ since it gives $M \square P_n$ is chromatic-choosable. It improves the bound from Theorem~\ref{thm: Borow} for any $H$ satisfying $\col(H)=\chi_{\ell}(H)$ and $\rho= \col(H)-1$ (examples of graphs where both of these conditions are satisfied include paths, cycles, complete graphs, and powers of paths) while $M$ can be any strong $k$-chromatic-choosable graph that satisfies the edge condition. Next we study the case when $\rho >1$.

In Section 5, we study how the list color function can be used to bound list chromatic number of certain Cartesian products.  One of our early results is:
\begin{thm*}
If $M$ is a strong $k$-chromatic-choosable graph, then
\[
\chi_{\ell}(M \square K_{1,s}) =
\begin{cases}
k & \text{if } s < P_{\ell}(M,k)\\
k+1 & \text{if } s \geq P_{\ell}(M,k).
\end{cases}
\]
\end{thm*}

It can be shown $ P_{\ell}(M,k) =  P(M,k)$ when $M$ is an odd cycle, complete graph, or the join of an odd cycle and complete graph.  The following question is open.

\vspace{5mm}
\noindent {\bf Question:} Is $P_\ell(G,k)=P(G,k)$ whenever $G$ is strongly chromatic-choosable?
\vspace{5mm}

When we can compute the exact value of $P_{\ell}(M,k)$ for some strong $k$-chromatic-choosable graph $M$, we can determine when $\chi_\ell(M \square K_{1,s})$ transitions from $k$ to $k+1$.  For example, $\chi_{\ell}(K_n \vee C_{2l+1} \square K_{1,s})$ transitions from $n+3$ to $n+4$ at $s=\frac{1}{3} (n+3)! (4^l-1)$.  Note that this gives results, including chromatic-choosable Cartesian products, where the second factor of the product, a star, is far from having a Hamilton path. We can also extend this theorem to get chromatic-choosability when the second factor is a subdivision of a star and the first factor is a strongly chromatic-choosable graph satisfying the edge condition.

Finally, we define $(M, \rho)$-Cartesian accommodating graphs and prove that the theorem from Section 4 remains true with the second factor replaced by this more general class of graphs. We prove sharpness, for all $\rho$, by constructing a $(M, \rho)$-Cartesian accommodating graph $H$, with $\chi_{\ell}(M \square H) = k+ \rho -1$, by a recursive construct $S_{M,B', \rho}$ that glues together $P_{\ell}(M,k+ \rho -2)$ disjoint copies of $S_{M,B', \rho -1}$ starting with $B'$, a subdivision of star $K_{1, P_{\ell}(M,k)-1}$. This allows us to give recursive constructions of large chromatic-choosable graphs.

\section{Strongly Chromatic-Choosable Graphs}\label{scc}

In this section we introduce the notion of strongly chromatic-choosable graphs. We refer to a list assignment, $L$, for graph $G$ as a \emph{bad k-assignment} for $G$ if $L$ is a $k$-assignment and $G$ is not $L$-colorable.  A list assignment, $L$, for a graph $G$ is called \emph{constant} if $L(v)$ is the same list for each $v \in V(G)$.  Now, we define a graph $G$ to be \emph{strong k-chromatic-choosable} if $\chi(G) = k$ and every bad $(k-1)$-assignment for $G$ is constant. One will note that strong 1-chromatic-choosable graphs are graphs with at least one vertex and no edges which is rather uninteresting.  Therefore, unless otherwise noted, we will focus our attention on strong $k$-chromatic-choosable graphs with $k \geq 2$.  It is easy to see that when $k \geq 2$, a graph, $G$, is strong $k$-chromatic-choosable if and only if $\chi(G) > k-1$ and and every bad $(k-1)$-assignment for $G$ is constant.  We will often use this characterization of strongly chromatic-choosable in the proofs below.

We introduce some terminology that will be useful here as well as later in the paper. Suppose that $G_1$ and $G_2$ are two arbitrary vertex disjoint graphs.  The \emph{disjoint union} of the graphs $G_1$ and $G_2$, denoted $G_1 + G_2$, is the graph with vertex set $V(G_1) \cup V(G_2)$ and edge set $E(G_1) \cup E(G_2)$.  The \emph{join} of the graphs $G_1$ and $G_2$, denoted $G_1 \vee G_2$, is the graph consisting of $G_1$, $G_2$, and additional edges added so that each vertex in $G_1$ is adjacent to each vertex in $G_2$.  It is well known that $\chi(G_1 \vee G_2)=\chi(G_1)+\chi(G_2)$. Dirac~\cite{D52} showed $G_1 \vee G_2$ is $(k_1+k_2)$-critical if and only if $G_i$ is $k_i$-critical for $i=1,2$.

\begin{pro} \label{pro: obvious} Suppose $G$ is a strong $k$-chromatic-choosable graph.  Then,
\\
(i)  $\chi_{\ell}(G)=k$ (i.e. $G$ is chromatic-choosable); \\
(ii)  If $L$ is a list assignment for $G$ with $|L(v)| \geq k-1$ for each $v \in V(G)$ and $L$ is not a constant $(k-1)$-assignment, then $G$ is $L$-colorable; \\
(iii)  $G \vee K_p$ is strong $(k+p)$-chromatic-choosable for any $p \in \N$; \\
(iv)  For any $v \in V(G)$, $\chi(G-\{v\}) \leq \chi_{\ell}(G - \{v\})<k$; \\
(v)  $k=2$ if and only if $G$ is $K_2$; \\
(vi)  $k=3$ if and only if $G$ is an odd cycle.
\end{pro}

\begin{proof}
For all 6 of the statements, we have that $k \geq 2$ and hence $\chi(G) > 1$.  So, $G$ has at least one edge, and we let $V(G)= \{v_1, \ldots, v_n \}$ where $n \geq 2$ throughout.

\par

For (i), suppose that $L$ is an arbitrary $k$-assignment for $G$.  We form a new list assignment, $L'$, for $G$ by deleting one color from $L(v_1)$ and one color from $L(v_2)$ so that $L'(v_1) \neq L'(v_2)$.  We then arbitrarily delete one color from each list associated with any remaining vertices.  Now, $L'$ is a non-constant $(k-1)$-assignment for $G$, and since $G$ is strong $k$-chromatic-choosable, there must be a proper $L'$-coloring for $G$.  This coloring is also a proper $L$-coloring for $G$, and we have that: $k = \chi(G) \leq \chi_{\ell}(G) \leq k.$

\par

For (ii), note that we are immediately done if $L$ is a $(k-1)$-assignment since in this case $L$ cannot be constant.  If $L$ is not a $(k-1)$-assignment we proceed as we did in the proof of statement (i), that is, we delete colors from lists so as to obtain a non-constant $(k-1)$-assignment.

\par

For (iii), one can follow the proof of a similar statement found in~\cite{ST09}.  We will present a slightly longer proof that is similar in flavor to some of the arguments we make later on in the paper.  Let $G'=G \vee K_p$ and suppose that $\{w_1, \ldots, w_p \}$ is the vertex set of the copy of $K_p$ joined to $G$ to form $G'$.  We know that $\chi(G')=\chi(G)+p=k+p$.  Now, suppose that $L$ is an arbitrary non-constant $(k+p-1)$-assignment for $G'$.  We will prove that there is a proper $L$-coloring for $G'$ in two cases: (1) There exists $w,u \in V(G)$ such that $L(w) \neq L(u)$ and (2) $L(v)=A$ for each $v \in V(G)$.  In (1) for each $i$ we color $w_i$ with $c_i$ so that $c_i \in L(w_i)$ and $c_i \neq c_j$ whenever $i \neq j$ (this is possible since each list contains at least $p$ colors).  Now, for each $v \in V(G)$ we let $L'(v)=L(v)-\{c_1, \ldots c_p \}$.  We note that $L'$ is a list assignment for $G$ with the property that $|L'(v)| \geq k-1$ for each $v \in V(G)$.  Since $L(w) \neq L(u)$, $L'$ is not a constant $(k-1)$-assignment.  Statement (ii) then implies that we can find a proper $L'$-coloring for $G$.  Thus, we can complete a proper $L$-coloring for $G'$.  For case (2), we note that since $L$ is non-constant, there must be some $c_1 \in \bigcup_{i=1}^p L(w_i)-A$.  Without loss of generality, suppose that $c_1 \in L(w_1)-A$.  We color $w_1$ with $c_1$.  Then, for each $i \geq 2$ we color $w_i$ with $c_i$ so that $c_i \in L(w_i)$ and $c_i \neq c_j$ whenever $i \neq j$.  Now, for each $v \in V(G)$ we let $L'(v)=L(v)-\{c_1, \ldots c_p \}$.  We note that $L'$ is a list assignment for $G$ with the property that $|L'(v)| \geq k$ for each $v \in V(G)$.  Statement (ii) then implies that we can find a proper $L'$-coloring for $G$, and this means we can complete a proper $L$-coloring for $G'$.

\par

For (iv), the first inequality is obvious.  So, we need only show the second inequality.  Suppose $v$ is an arbitrary vertex of $G$.  Let $G' = G - \{v\}$, and suppose that $L'$ is an arbitrary $(k-1)$-assignment for $G'$.  Now, let $L$ be the $(k-1)$-assignment for $G$ obtained by letting $L(x) = L'(x)$ for each $x \in V(G')$ and letting $L(v)$ be a list of $(k-1)$ colors such that $L(v) \neq L(y)$ for some $y \in V(G')$.  Then, $L$ is a non-constant $(k-1)$-assignment for $G$, and we know there is a proper $L$-coloring, $c$, for $G$.  Then, we define a coloring, $c'$, for $G'$ by letting $c'(x)=c(x)$ for each $x \in V(G')$.  Clearly $c'$ is a proper $L'$-coloring for $G'$.  Since $L'$ was arbitrary, it follows that $\chi_{\ell}(G') \leq k-1 < k.$

\par

For (v), note that if $G=K_2$, then $G$ is strong 2-chromatic-choosable since $\chi(G)=2$, and $G$ is $L$-colorable whenever $L$ is a non-constant 1-assignment for $G$.  Conversely, if $G$ is strong 2-chromatic-choosable, the fact that $G$ must be a copy of $K_2$ follows from Statement (iv). The proof of Statement (vi) is similar to the proof of Statement (v).
\end{proof}

As discussed in Section 1.2, this notion of strongly chromatic-choosable graphs is in fact an extension of an older notion of criticality defined by Stiebitz, Tuza, and Voigt~\cite{ST09}. They define a graph $G$ to be \emph{strong k-critical} if $G$ is $k$-critical and every bad $(k-1)$-assignment for $G$ is constant.  One will immediately note that every strong $k$-critical graph is strong $k$-chromatic-choosable. Moreover, a graph is strong $k$-critical if and only if the graph is $k$-critical and strong $k$-chromatic-choosable.  The strong $k$-critical graphs are the same as the strong $k$-chromatic-choosable graphs when $k=2,3$. We will show below that there exist strong $k$-chromatic-choosable graphs that are not strong $k$-critical for all $k \geq 4$.

Steibitz et al.~\cite{ST09} show that the complete graph, $K_k$, is strong $k$-critical, and that odd cycles are strong 3-critical. Like Proposition~\ref{pro: obvious}(iii) above, they also showed that: If $G$ is a strong $k$-critical graph, then the graph $G' = G \vee K_p$ is strong $(k+p)$-critical. They also construct two other types of strongly critical graphs. For $k \geq 3$ a \emph{$D_k$-graph} is a graph, $G$, whose vertex set consists of three non-empty pairwise disjoint sets $X$, $Y_1$, and $Y_2$ with $|Y_1|+|Y_2|=|X|+1=k-1$ and two additional vertices, $x_1$ and $x_2$, such that $X$ and $Y_1 \cup Y_2$ are cliques in $G$ not joined by any edge and $x_i$ is adjacent to each vertex in $X \cup Y_i$ for $i=1,2$.  We write $G=D_k(X,Y_1,Y_2,x_1,x_2)$.  Similarly, for $k \geq 3$ a \emph{$E_k$-graph} is a graph, $G$, whose vertex set consists of four non-empty pairwise disjoint sets $X_1$, $X_2$, $Y_1$, and $Y_2$ with $|Y_1|+|Y_2|=|X_1|+|X_2|=k-1$ and $|X_2|+|Y_2| \leq k-1$, and one additional vertex, $z$, such that $X=X_1 \cup X_2$ and $Y=Y_1 \cup Y_2$ are cliques in $G$, $z$ is adjacent to each vertex in $X_1 \cup Y_1$, and $x \in X$ is adjacent to $y \in Y$ if and only if $x \in X_2$ and $y \in Y_2$.  We write $G=E_k(X_1,X_2,Y_1,Y_2,z)$.  See examples 7 and 8 of~\cite{ST09} for a proof of: For $k \geq 3$, all $D_k$ graphs and $E_k$ graphs are strong $k$-critical.

\subsection{Constructions}

Let's start with a simple construction based on the Dirac-Haj{\'o}s construction (see~\cite{SS89}). Specifically, suppose that $G_1$ and $G_2$ are vertex disjoint copies of $K_k$ with $k \geq 3$, $V(G_1)=\{v_1, \ldots, v_k \}$, and $V(G_2)=\{u_1, \ldots, u_k \}$.  For $i$ satisfying $1 \leq i \leq k-2$, we form a copy, $G$, of $K_k \Delta^{(i)} K_k$ by identifying $v_1$ with $u_1$, $v_2$ with $u_2$, $\ldots$, and $v_i$ with $u_i$, deleting the edges $v_1v_{i+1}$ and $u_1u_{i+1}$, and adding the edge $u_{i+1}v_{i+1}$.  One should note that for $i \geq 2$, $G$ is a copy of $(K_{k-i+1} \Delta^{(1)} K_{k-i+1}) \vee K_{i-1}$.  Also, for $i=1$, $G=E_k(\{v_3, \ldots, v_k \}, \{v_2\}, \{u_3, \ldots, u_k \}, \{u_2\}, u_1)$.  So, $G$ is strong $k$-critical when $i=1$. Then, applying the join procedure, $K_k \Delta^{(i)} K_k$ is strong $k$-critical when $i \geq 2$.

So we have several large families of graphs, such as odd cycles, complete graphs, $D_k$ graphs, $E_k$ graphs, and $K_k \Delta^{(i)} K_k$, that are strongly critical graphs, and hence, are also strongly chromatic-choosable. We will now focus on constructing some nontrivial examples of strongly chromatic-choosable graphs that are not strongly critical.  We begin with a lemma that is a useful tool for such constructions.

\begin{lem} \label{lem: construct}  Let $G$ be a strong $k$-chromatic-choosable graph. Let $A, B \subseteq V(G)$ such that $A \cup B = V(G)$ and $C=A \cap B$ with $|A|, |B| > |C|$, $0 <|C| \leq 3$ when $k$ is even and $0 < |C| \leq 4$ when $k$ is odd.  Form $G'$ by adding vertices $u$ and $s$ to $G$, and edges so that $u$ is adjacent to every vertex in $A$ and $s$ is adjacent to every vertex in $B$.  If $\chi(G') > k$, then $G'$ is strong $(k+1)$-chromatic-choosable.
\end{lem}

Note that for this lemma we must have $k \geq 3$. This is because finding such an $A$ and $B$ would be impossible in the case that $k=2$ since $G$ would only have two vertices in this case by Proposition~\ref{pro: obvious} (v).

\begin{proof}  For this proof we will let $C=\{v_1, \ldots, v_m \}$.  Suppose that $L$ is an arbitrary non-constant $k$-assignment for $G'$.  In order to show that $G'$ is strong $(k+1)$-chromatic-choosable, we must show $G'$ is $L$-colorable.  We know that $G=G'-\{u, s \}$.  Now, we know that exactly one of the following two cases must hold: (1) $L(u) \cap L(s) \neq \emptyset$ or (2) $L(u) \cap L(s) = \emptyset$.  We will construct a proper $L$-coloring of $G'$ in each of these cases.  In both cases our general strategy is the same: first color $u$ and $s$, then use the fact that $G$ is strong $k$-chromatic-choosable to complete a proper coloring.

\par

For case (1) suppose that $c_1 \in L(u) \cap L(s)$.  We color both $u$ and $s$ with $c_1$.  Now, for each vertex $v \in V(G)$, let
$$L'(v) = L(v) - \{c_1\}.$$
\noindent We notice that $|L'(v)| \geq k-1$ for each $v \in V(G)$.  Also, we know that $L'$ is either a constant $(k-1)$-assignment for $G$ or it is not.  In the case that $L'$ is not a constant $(k-1)$-assignment for $G$, we know that we can complete a proper $L$-coloring for $G'$ by Proposition~\ref{pro: obvious} Statement (ii).  Now, consider the case that $L'$ is a constant $(k-1)$-assignment for $G$.  In this case, $L(v)$ must be the same for each $v \in V(G)$.  Let $W=L(v)$ for each $v \in V(G)$.  Since $L$ is a non-constant $k$-assignment, this means that $L(u)$ or $L(s)$ must be different from $W$.  Without loss of generality suppose $c_2 \in L(u)-W$.  We recolor vertex $u$ with $c_2$.  Now, for each vertex $v \in V(G)$, let
\[ L''(v) = \begin{cases}
      L(v) - \{c_1\} & \textrm{ if $v$ is a neighbor of $s$ in $G$} \\
      L(v) & \textrm{ if $v$ is not a neighbor of $s$ in $G$} \\
   \end{cases} \]
\noindent  We note that $|L''(v)| \geq k-1$ for each $v \in V(G)$.  Since $|A| \geq |C|+1$, there is at least one vertex in $V(G)$ that is not a neighbor of $s$.  So, we have that $L''$ is not a constant $(k-1)$-assignment for $G$.  Thus, there is a proper $L''$ coloring for $G$ by Proposition~\ref{pro: obvious} Statement (ii), and we can complete a proper $L$-coloring for $G'$.

\par

Now, we turn our attention to case (2).  For this case let $I=\bigcap_{i=1}^m L(v_i)$.  We consider two sub-cases: (a) $|I| \geq k-1$ and (b) $|I| < k-1$ (\emph{Note:} when $m=1$, $|I|=k$ and we are in sub-case (a)).  For (a), suppose $U=\{c_1, \ldots, c_{k-1}\} \subseteq I$.  Since $L(u)$ and $L(s)$ are disjoint and have at least $k$ elements we know that each of these lists contain a color not in $U$, and at least one of the lists contains at least two colors not in $U$.  Without loss of generality suppose that $c', c'' \in L(u)-U$ and $c''' \in L(s)-U$.  Now, for each $v \in V(G)$ let
\[ L'(v) = \begin{cases}
      L(v) - \{c'\} & \textrm{ if $v \in A-C$} \\
      L(v) - \{c'''\} & \textrm{ if $v \in B-C$} \\
			L(v) - \{c',c'''\} & \textrm{if $v \in C$}\\
   \end{cases} \]
\noindent and
\[ L''(v) = \begin{cases}
      L(v) - \{c''\} & \textrm{ if $v \in A-C$ } \\
      L(v) - \{c'''\} & \textrm{ if $v \in B-C$ } \\
			L(v) - \{c'',c'''\} & \textrm{if $v \in C.$}\\
   \end{cases} \]
\noindent We note that $|L'(v)| \geq k-1$ and $|L''(v)| \geq k-1$ for each $v \in V(G)$.  Moreover, since $|A| \geq |C|+1$, we know that at least one of $L'$ and $L''$ is not a constant $(k-1)$-assignment.  Without loss of generality suppose that $L'$ is not a constant $(k-1)$-assignment for $G$.  Then, color $u$ with $c'$ and $s$ with $c'''$, and since there is a proper $L'$ coloring for $G$ by Proposition~\ref{pro: obvious} Statement (ii), we can complete a proper $L$-coloring for $G'$.

\par

Now, we turn our attention to sub-case (b).  Suppose that $L(u)= \{c_1, \ldots, c_k \}$ and $L(s)= \{c_{k+1}, \ldots, c_{2k} \}$.  We notice that if we are to find an $L$-coloring of $G'$, then we have $k^2$ possible ways to color $u$ and $s$.  We can think of each of these $k^2$ possibilities as an ordered pair in the set $P=\{(c_i,c_j)|i \in \{1,\ldots,k\}, j \in \{k+1,\ldots,2k\} \}$.  We refer to each of the $k^2$ possibilities in this set as a \emph{color pair}.  We say that a list of colors \emph{contains a color pair} if it contains both colors that make up the coordinates of the pair.  Now, we claim that there must be some color pair that is not contained in any of the lists: $L(v_1), \ldots, L(v_m)$.  To see why this is so, note that for any $i$ ($1 \leq i \leq m$) $L(v_i)$ contains at most $\lceil \frac{k}{2} \rceil \lfloor \frac{k}{2} \rfloor$ color pairs.  We note that $\lceil \frac{k}{2} \rceil \lfloor \frac{k}{2} \rfloor < \frac{k^2}{3}$ when $k$ is even and $\lceil \frac{k}{2} \rceil \lfloor \frac{k}{2} \rfloor < \frac{k^2}{4}$ when $k$ is odd.  So, since $|C| \leq 3$ when $k$ is even and $|C| \leq 4$ when $k$ is odd, we have that there must be some color pair that is not contained in any of the lists: $L(v_1), \ldots, L(v_m)$.  Without loss of generality suppose that $(c_1,c_{k+1})$ is not contained in any of these lists.  We color $u$ with $c_1$ and $s$ with $c_{k+1}$ and for each $v \in V(G)$ we let
\[ L'(v) = \begin{cases}
      L(v) - \{c_1\} & \textrm{ if $v \in A-C$} \\
      L(v) - \{c_{k+1}\} & \textrm{ if $v \in B-C$} \\
			L(v) - \{c_1,c_{k+1}\} & \textrm{if $v \in C$.}\\
   \end{cases} \]
\noindent Since no list associated with a vertex in $C$ contains the color pair $(c_1,c_{k+1})$, we have that $|L'(v)| \geq k-1$ for each $v \in V(G)$.  Also, since $|I| < k-1$, we know that $L'$ is not a constant $(k-1)$-assignment.   Thus, there is a proper $L'$ coloring for $G$ by Proposition~\ref{pro: obvious} Statement (ii), and we can complete a proper $L$-coloring for $G'$.
\end{proof}

We will now use this lemma to produce some examples of strong $k$-chromatic-choosable graphs.

\begin{pro} \label{pro: addtoodd}
Suppose $C$ is an odd cycle with vertices (in cyclic order): $v_1, v_2, v_3, v_4, \ldots, v_{2l+1}$ where $l \geq 2$.  Suppose $m \in \N$ is such that $m \leq 2l-2$.  We construct the graph $G_{l,m,1}$ as follows: Add vertices $u_1$ and $s_1$ to $C$, and add edges so that $u_1$ is adjacent to each vertex in $\{v_j | 1 \leq j \leq 2+m \}$ and so that $s_1$ is adjacent to each vertex in $V(C)-\{v_1,v_2 \}$.  Then, $G_{l,m,1}$ is strong $4$-chromatic-choosable whenever $m \leq 4$.
\end{pro}

At this point it may seem strange that we attach an additional parameter of ``1" to $G_{l,m,1}$.  The reason for this will be made clear shortly.

\begin{proof}
Throughout this proof we assume $m \leq 4$.  We first note that for any proper coloring of $C$, the path $P_1$ given by $v_1, v_2, v_3$ or the path $P_2$ given by $v_3, v_4, \ldots, v_{2l+1}$ must be colored with at least three colors.  To see why this is so note that if $P_1$ and $P_2$ were both colored with only two colors $v_1$ and $v_{2l+1}$ would receive the same color as $v_3$.  So, we have that $\chi(G_{l,m,1}) > 3$.  Now, let $A= \{v_j | 1 \leq j \leq 2+m \}$ and $B=V(C)-\{v_1,v_2 \}$.  We note that $A \cup B = V(C)$, $|A \cap B| = m \leq 4$.  Also, since $v_1 \notin A \cap B$, $|A| > |A \cap B|$.  Since $m \leq 2l-1$, we know $v_{2l+1} \notin A \cap B$ and $|B| > |A \cap B|$.  Thus, Lemma~\ref{lem: construct} immediately implies that $G_{l,m,1}$ is strong $4$-chromatic-choosable when $m \leq 4$.
\end{proof}

It is easy to verify that $G_{l,1,1}$ is $4$-critical.  Thus, $G_{l,1,1}$ is strong $4$-critical.  Since $G_{l,2,1}$, $G_{l,3,1}$, and $G_{l,4,1}$ contains $G_{l,1,1}$ as a subgraph, we have that these three graphs are not $4$-critical.  Thus, we have our first examples of strongly chromatic-choosable graphs that are not strongly critical.  For $p \geq 1$ and $2 \leq m \leq 4$, note that Proposition~\ref{pro: obvious} Statement (iii) implies that $G_{l,m,1} \vee K_p$ is strong $(4+p)$-chromatic-choosable, yet $G_{l,m,1} \vee K_p$ is not $(4+p)$-critical.  So, there exist strong $k$-chromatic-choosable graphs that are not strong $k$-critical for $k \geq 4$.  We will now illustrate one more application of Lemma~\ref{lem: construct} by inductively extending the idea of Proposition~\ref{pro: addtoodd}.  We postpone its proof until the appendix.

\begin{pro} \label{pro: addaddtoodd} For $k \in \N$ and $m \in \{1,2,3\}$ we construct $G_{l,m,k}$ inductively as follows.  For $k=1$, $G_{l,m,k}$ is the graph constructed in the statement of Proposition~\ref{pro: addtoodd}.  For $k \geq 2$ we construct $G_{l,m,k}$ from $G_{l,m,k-1}$ as follows.  We add vertices $u_k$ and $s_k$ to $G_{l,m,k-1}$ and we add edges so that $u_k$ is adjacent to $\{u_j | 1 \leq j \leq k-1 \} \cup \{v_j | 1 \leq j \leq 2+m \}$ and so that $s_k$ is adjacent to $\{s_j | 1 \leq j \leq k-1 \} \cup (V(C)-\{v_1,v_2 \})$.  Then, $G_{l,m,k}$ is strong $(3+k)$-chromatic-choosable.  Moreover, $G_{l,m,k}$ is not strong $(3+k)$-critical when $m=2,3$.
\end{pro}

{\bf Historical Remarks:}  We have defined strongly chromatic-choosable graphs using the language of list coloring.  Researchers have studied questions closely related to list coloring via amenable colorings (see~\cite{BK90} and~\cite{MR97}).  In fact, we could have defined strong $k$-chromatic-choosable graph by using the language of amenable colorings for $k \geq 2$.  We briefly explain how this would work by following the definitions in~\cite{MR97}.  Suppose that $G$ is a graph and suppose that $m>j>0$.  For every $v \in V(G)$, let $R(v)$ be a $j$-element subset of $\{1,2, \ldots, m \}$, and suppose that $R$ is not constant.  Such an $R$ is called a \emph{$(j,m)$-restraint}.  An \emph{R-amenable} coloring of $G$ is a proper coloring, $f: V(G) \rightarrow \{1,2, \ldots, m \}$, such that $f(v) \in \{1,2, \ldots, m \} - R(v)$ for each $v \in V(G)$.  We say that $G$ is \emph{$(j,m)$-amenable} if there is an $R$-amenable coloring of $G$ for every possible $(j,m)$-restraint, $R$.  For $k \in \N$ we let $J_k(G)$ be a nonnegative integer or $\infty$ such that if $j$ is a positive integer, $G$ is $(j,j+k)$-amenable if and only if $j \leq J_k(G)$.  Now, we have that $G$ is strong $k$-chromatic-choosable if and only if $\chi(G) > k-1$ and $J_{k-1} (G) = \infty$.

\par

In July 2001 Mohar~\cite{M01} made an interesting conjecture that can be stated in terms of strong chromatic-choosability. Suppose $G$ is a graph with maximum degree $\Delta(G)$.  We say $G$ is \emph{$k$-edge-colorable} if its edges can be colored with $k$ colors such that any two incident edges receive different colors. We say that $G$ is \emph{$(\Delta(G)+1)$-edge-critical} if it is $(\Delta(G)+1)$-edge-colorable but every subgraph of $G$ in which at least one edge of $G$ is not present is $\Delta(G)$-edge-colorable.  In terms of strong chromatic-choosability Mohar's conjecture may be stated as follows.

\begin{conj}[\cite{M01}] \label{conj: Mohar}
Suppose $G$ is a $(\Delta(G)+1)$-edge-critical graph, and let $L(G)$ be the line graph of $G$.  Then, $L(G)$ is strong $(\Delta(G)+1)$-chromatic-choosable.
\end{conj}

\section{Non-unique List Colorability and the Edge Condition} \label{unique}

We begin with a standard definition.  Suppose $G$ is a graph, and consider a function $f: V(G) \rightarrow \N$.  We say that $G$ is \emph{$f$-choosable} if $G$ is $L$-colorable whenever $|L(v)| \geq f(v)$ for each $v \in V(G)$.  Akbari, Mirrokni, and Sadjad~\cite{AM06} extended some of the ideas of Alon and Tarsi~\cite{AN92} and proved the following two results.

\begin{thm}[\cite{AM06}] \label{thm: akbari}
Suppose that $G$ is a graph with $n$ vertices and $m$ edges and assume that $f: V(G) \rightarrow \N$ is a function with $\sum_{v \in V(G)} f(v)=m+n$.  If there is a list assignment, $L$, for $G$ such that $|L(v)|=f(v)$ for each $v \in V(G)$ and there is a unique proper $L$-coloring for $G$, then $G$ is $f$-choosable.
\end{thm}

\begin{cor}[\cite{AM06}] \label{cor: akbari}
Suppose that $G$ is a graph with $n$ vertices and $m$ edges and assume that $f: V(G) \rightarrow \N$ is a function with $\sum_{v \in V(G)} f(v) > m+n$.  Then, there is no list assignment, $L$, for $G$ such that $|L(v)|=f(v)$ for each $v \in V(G)$ and $G$ has a unique proper $L$-coloring.
\end{cor}

Generally speaking, the key to proving some of our list coloring results is that certain types of strong $k$-chromatic-choosable graphs have at least two proper colorings for non-constant $(k-1)$-assignments.

\begin{lem} \label{lem: notunique}
Suppose that $G$ is a strong $k$-chromatic-choosable graph with $n$ vertices and $m$ edges.  Suppose that $L$ is a list assignment for $G$ such that $|L(v)| \geq k-1$ for each $v \in V(G)$ and $L$ is not a constant $(k-1)$-assignment for $G$.  If $m \leq n(k-2)$, then there are at least two proper $L$-colorings for $G$.
\end{lem}

\begin{proof}  Suppose that $L$ is a non-constant $(k-1)$-assignment for $G$ and $m \leq n(k-2)$.  It suffices to show that there are at least two proper $L$-colorings for $G$.  By the definition of strong $k$-chromatic-choosable we know that there is a proper $L$-coloring for $G$ which we will call $c$.  For the sake of contradiction suppose that $c$ is the unique proper $L$-coloring for $G$.  We assume $f: V(G) \rightarrow \N$ is given by the rule $f(v)=k-1$.  Note that if we have $m+n < n(k-1)$, Corollary~\ref{cor: akbari} implies that there is not a unique proper $L$-coloring for $G$ which is a contradiction.  So, we may assume $m+n = n(k-1)$.  Since $c$ is the unique proper $L$-coloring for $G$, we have that $G$ is $f$-choosable by Theorem~\ref{thm: akbari}.  This implies $\chi(G) \leq \chi_{\ell}(G) \leq k-1$.  This is a contradiction, and the proof is complete.\end{proof}

For the remainder of this paper, we say that a strong $k$-chromatic-choosable graph, $G$, with $n$ vertices and $m$ edges satisfies the \emph{edge condition} if $m \leq n(k-2)$.  Notice that the edge condition is violated for all strong $k$-chromatic-choosable graphs with $k \leq 2$.  It is easy to verify that odd cycles and complete graphs on at least three vertices satisfy the edge condition.  This means that all strong $3$-chromatic-choosable graphs satisfy the edge condition.  The strong $k$-chromatic-choosable graphs described in Section 2 all satisfy the edge condition.

\begin{pro} \label{pro: edgeDkandEk}
For $k \geq 3$, $D_k$ and $E_k$ graphs satisfy the edge condition.
\end{pro}

The proof follows by careful calculation.  The two recipes for construction that we gave in Section 2 also preserve the edge condition.


\begin{pro} \label{pro: edgeconstruct}
Suppose that $G$ is a strong $k$-chromatic-choosable graph with $n$ vertices and $m$ edges that satisfies the edge condition.  Then, the following two statements hold.
\\
\noindent (i) For $p \geq 1$, $G \vee K_p$ satisfies the edge condition, and
\\
\noindent (ii)  If $G'$ is constructed from $G$ so that it satisfies the hypotheses of Lemma~\ref{lem: construct}, then $G'$ satisifes the edge condition.
\end{pro}

\begin{proof}
We know that $m \leq n(k-2)$.  For (i),  we know from Proposition~\ref{pro: obvious} Statement (iii) that $G \vee K_p$ is strong $(k+p)$-chromatic-choosable, and $G \vee K_p$ has $n+p$ vertices and $m+np+ \frac{p(p-1)}{2}$ edges.  It is easy to show that $\frac{p(p+1)}{2} \leq kp+p^2-p$.  Then, if we combine this fact with the fact that $m \leq n(k-2)$, it is easy to obtain $m+ np + \frac{p(p-1)}{2} \leq (n+p)(k+p-2)$.
\par
For (ii), we know from Lemma~\ref{lem: construct} that $G'$ is strong $(k+1)$-chromatic-choosable and $k \geq 3$.  Moreover, $G'$ has $n+2$ vertices and at most $m+n+4$ edges.  Since $k \geq 3$ and $m \leq n(k-2)$, we can easily obtain: $m+n+4 \leq (n+2)(k-1)$.
\end{proof}

The above proposition shows that any strongly chromatic-choosable graph constructed using Proposition~\ref{pro: obvious} Statement (iii) or Lemma~\ref{lem: construct} from a graph satisfying the edge condition satisfies the edge condition. Thus, Lemma~\ref{lem: notunique} applies to such graphs. Also note Proposition~\ref{pro: joincomplete} below shows that even a strongly chromatic-choosable graph that does not satisfy the edge condition can be made to satisfy the edge condition by taking a join with a large enough complete graph.

We now show that there are examples of strongly chromatic-choosable graphs that do not satisfy the edge condition.  In order to do this, we will show that Lemma~\ref{lem: construct} can be extended in the case where our starting graph is an odd cycle.

\begin{lem} \label{lem: constructodd} Let $G$ be an odd cycle $C_{2l+1}$. Suppose we can find sets $A, B \subseteq V(G)$ such that $A \cup B = V(G)$. Let $C=A \cap B$ and suppose  $ 0 < |C| \leq 8$, and $|A|, \;|B| > |C|$.  Form $G'$ by adding vertices $u$ and $s$ to $G$, and add edges so that $u$ is adjacent to every vertex in $A$ and $s$ is adjacent to every vertex in $B$.  If $\chi(G') > 3$, then $G'$ is strong $4$-chromatic-choosable.
\end{lem}

\begin{proof}  The proof is postponed to the Appendix to save space.
\end{proof}

One will note that it is easy to construct examples where $u$ and $s$ have at least 10 neighbors in common and the resulting graph is not strong 4-chromatic-choosable.  This leads to an interesting question: In Lemma~\ref{lem: constructodd} can the condition $|C| \leq 8$ be replaced with $|C| \leq 9$?  We suspect that the answer is yes.  More importantly, using the notation of Proposition~\ref{pro: addtoodd}, Lemma~\ref{lem: constructodd} shows that $G_{l,5,1}$, $G_{l,6,1}$, $G_{l,7,1}$, and $G_{l,8,1}$ are all strong $4$-chromatic-choosable.  We note that $G_{l,5,1}$, $G_{l,6,1}$, $G_{l,7,1}$, and $G_{l,8,1}$ all do not satisfy the edge condition since $G_{l,m,1}$ has $4l+m+2$ edges and $2l+3$ vertices.  Thus, the are infinitely many strong 4-chromatic-choosable graphs that do not satisfy the edge condition.  Also, $G_{l,8,1} \vee K_1$ is a strong 5-chromatic-choosable graph not satisfying the edge condition.

\par

There also exist strongly critical graphs that do not satisfy the edge condition.  Specifically, a tedious argument shows that $C_5 \vee C_{2n+1}$ is strong 6-critical for $n=1,2,3,4,5$~\cite{KM216}.  It is easy to see that $C_5 \vee C_9$ and $C_5 \vee C_{11}$ do not satisfy the edge condition.  Then, $ (C_5 \vee C_{11}) \vee K_1$ is strong 7-critical and does not satisfy the edge condition.  Note this also implies that there are strong $k$-chromatic-choosable graphs violating the edge condition for all $k$ satisfying $4 \leq k \leq 7$.

\par

We suspect that there are infinitely many strong $k$-chromatic-choosable graphs that do not satisfy the edge condition for $k \geq 8$.  However, whenever we have a strongly chromatic-choosable graph, $G$, that does not satisfy the edge condition, we can obtain a strongly chromatic-choosable graph from $G$ that satisfies the edge condition by taking the join of $G$ with a sufficiently large complete graph.  The following proposition makes this idea precise.

\begin{pro} \label{pro: joincomplete}  Suppose that $G$ is a strong $k$-chromatic-choosable graph with $n$ vertices and $m$ edges that does not satisfy the edge condition.  Let $d= m - n(k-2) > 0$.  Then, if $p$ is such that $d \leq p(2k+p-3)/2$, $G \vee K_p$ is a strong $(k+p)$-chromatic-choosable graph that satisfies the edge condition.
\end{pro}

\begin{proof}
By Proposition~\ref{pro: obvious} Statement (iii), we know that $G \vee K_p$ is a strong $k+p$-chromatic-choosable.  Also, $G \vee K_p$ has $m + \frac{p(p-1)}{2} + np$ edges and $n+p$ vertices.  It is easy to show that $d \leq p(2k+p-3)/2$ implies $m + \frac{p(p-1)}{2} + np \leq (n+p)(k+p-2)$ which shows that $G \vee K_p$ satisfies the edge condition.
\end{proof}

So, $G_{l,m,1} \vee K_p$ is strong $(4+p)$-chromatic-choosable and satisfies the edge condition for $m \leq 7$ and $p \geq 1$.  Moreover, $G_{l,8,1} \vee K_p$ is strong $(4+p)$-chromatic-choosable and satisfies the edge condition for $p \geq 2$.

\section{List Coloring Cartesian Product of Graphs}\label{mainresult}
In this section we will prove the following theorem.

\begin{thm} \label{thm: mainresult}
Let $M$ be a strong $k$-chromatic-choosable graph that satisfies the edge condition, and $H$ be a graph that contains a Hamilton path, $w_1, w_2, \ldots, w_m$, such that $w_i$ has at most $\rho \ge 1$ neighbors among $w_1, \ldots, w_{i-1}$.  Let $G= M \square H$.  Then, $\chi_{\ell}(G) \leq k+ \rho -1$.
\end{thm}

We present a lemma that will immediately yield this theorem.  Our proof of the lemma will be by induction, and we will load the induction hypothesis so that we prove something stronger.  Before stating the lemma, we introduce some terminology.  Suppose that $G$ is a graph and $U \subseteq V(G)$, the \emph{subgraph of $G$ induced by $U$} is the graph with $U$ as its vertex set and all the edges in $E(G)$ with both endpoints in $U$, and we write $G[U]$.  Suppose that $H$ is a graph that contains a Hamilton path, $w_1, w_2, \ldots, w_m$, and  $G= M \square H$.  For $j=1,2, \ldots, m$ let $V_j$ be the set of vertices in $V(G)$ with second coordinate $w_j$.  We refer to $G[V_j]$ as the \emph{$j^{th}$ copy of $M$ in $G$}.

\begin{lem} \label{lem: main}
Let $M$ be a strong $k$-chromatic-choosable graph that satisfies the edge condition, and  $H$ be a graph that contains a Hamilton path, $w_1, w_2, \ldots, w_m$, such that  $w_i$ has at most $\rho \ge 1$ neighbors among $w_1, \ldots, w_{i-1}$. Let $G= M \square H$.  Suppose that $L$ is an arbitrary $(k+\rho-1)$-assignment for $G$.  There exist two proper $L$-colorings for $G$, $c_1$ and $c_2$, with the property that there exists a vertex, $v$, in the $m^{th}$ copy of $M$ in $G$ such that $c_1(v) \neq c_2(v)$, and for any vertex $u$ not in the $m^{th}$ copy of $M$ in $G$ we have that $c_1(u)=c_2(u)$.
\end{lem}

\begin{proof}
We will prove this by induction on $m$, the number of vertices in $H$.  We also suppose that $V(M)=\{v_1, \ldots, v_n \}$.  We know by Lemma~\ref{lem: notunique} that the claim holds for any $\rho \geq 1$ when $m=1$.

\par

The general idea for the induction is as follows.  As we color the copies of $M$ inductively, we will always have fixed colors for the first $m-2$ copies of $M$ in $G$.  We then use the inductive hypothesis to possibly modify how we will color the $(m-1)^{st}$ copy of $M$ in $G$ to make our coloring for the final copy work out.  We now present the details for the induction step.

\par

Now, suppose that $m \geq 2$ and the desired result holds for all natural numbers less than $m$.  By the induction hypothesis we know that there are two proper $L$-colorings of $G'=M \square (H-w_m)$ (when $L$ is restricted to $G'$), $c_1$ and $c_2$, with the property that there exists a vertex, $v$, in the $(m-1)^{st}$ copy of $M$ in $G'$ such that $c_1(v) \neq c_2(v)$, and for any vertex $u$ not in the $(m-1)^{st}$ copy of $M$ in $G'$ we have that $c_1(u)=c_2(u)$.  The strategy of the proof is to extend $c_1$ or $c_2$ into two proper $L$-colorings for $G$ that have the desired properties.

\par

Now for $i=1, \ldots, n$, let:
$$A_i = \{v \in V(G) | \text{$v=(v_i,w_j)$ with $j < m-1$ and $v$ is adjacent to $(v_i,w_m)$ in $G$} \}.$$
\noindent Intuitively, $A_i$ consists of the neighbors of $(v_i,w_m)$ that are in the first $(m-2)$ copies of $M$ in $G$.  We immediately note that $|A_i| \leq (\rho-1)$ for each $i$.  For $i=1, \ldots, n$ and $j=1,2$ let:
$$B_{i,j} = \{c_j(v) | v \in A_i \}.$$
\noindent We see that $|B_{i,j}| \leq (\rho-1)$ for each $i$ and $j$.  Also by the induction hypothesis we know that for each $i$, $B_{i,1}=B_{i,2}$.  So, we let $B_i = B_{i,1}=B_{i,2}$.  Now, for each $i$ let:
$$L'(v_i,w_m) = L(v_i,w_m) - (B_i \cup \{c_1(v_i,w_{m-1}) \}) \; \text{and} \; L''(v_i,w_m) = L(v_i,w_m) - (B_i \cup \{c_2(v_i,w_{m-1}) \}).$$
\noindent We immediately note that $|L'(v_i,w_m)| \geq k-1$ and $|L''(v_i,w_m)| \geq k-1$ for each $i$.  Now, we claim that $L'$ or $L''$ is not a constant $(k-1)$-assignment for the $m^{th}$ copy of $M$ in $G$.

\par

Suppose for the sake of contradiction that both $L'$ and $L''$ are constant $(k-1)$-assignments for the $m^{th}$ copy of $M$ in $G$.  Since both $L'$ and $L''$ are constant $(k-1)$-assignments, we know that for each $i$ we must have deleted exactly $\rho$ colors when we formed $L'(v_i,w_m)$ and $L''(v_i,w_m)$.  We note that $|B_i|$ is at most $\rho-1$ for each $i$. So, in order for $\rho$ colors to be deleted  when $L'(v_i,w_m)$ and $L''(v_i,w_m)$ are formed, for each $i$ we must have:
$$|B_i|=\rho-1 \; \text{and} \; |B_i \cup \{c_2(v_i,w_{m-1}) \}|=|B_i \cup \{c_1(v_i,w_{m-1}) \}|=\rho.$$
\noindent Moreover, we know that for each $i$, $B_i$ must be a subset of $L(v_i,w_m)$ which implies
$$|L(v_i,w_m)-B_i|=k.$$
\noindent Furthermore, since there is a vertex $v$ in the $(m-1)^{st}$ copy of $M$ in $G$ such that $c_1(v) \neq c_2(v)$, we know that there is an $i$ such that $L''(v_i,w_m) \neq L'(v_i,w_m)$.  Since $L'$ and $L''$ are constant $(k-1)$-assignments, this implies that $L''(v_i,w_m) \neq L'(v_i,w_m)$ for all $i$.  Since the union of two distinct $(k-1)$ element subsets of a $k$ element set must equal the $k$ element set,  we have that for each $i$, $L(v_i,w_m)-B_i=L''(v_i,w_m) \cup L'(v_i,w_m)$.  Thus,
$$  L(v_1,w_m)-B_1=L(v_2,w_m)-B_2= \cdots = L(v_{n},w_m)-B_n.$$
\noindent This means that in order for $L'$ to be a constant $(k-1)$-assignment, $c_1$ assigned the same color to every vertex in the $(m-1)^{st}$ copy of $M$ in $G$. This contradicts the fact that $c_1$ is a proper $L$-coloring of $G'$ since $\chi(M)=k >1$.  Thus, $L'$ or $L''$ is not a constant $(k-1)$-assignment for the $m^{th}$ copy of $M$ in $G$.

\par

Without loss of generality, suppose that $L'$ is not a constant $(k-1)$-assignment.  We know by Lemma~\ref{lem: notunique} that we can find at least 2 distinct proper $L'$-colorings of the $m^{th}$ copy of $M$ in $G$.  So, using these distinct $L'$-colorings, we can extend $c_1$ into 2 distinct proper $L$-colorings of $G$.  We know that these two distinct $L$-colorings will satisfy the needed conditions since the vertices not in the $m^{th}$ copy of $M$ in $G$ will be colored as they are by $c_1$ in both colorings.  This completes the induction step, and we are done.
\end{proof}

There are a couple of remarks worth making regarding Theorem~\ref{thm: mainresult}.  It is an improvement on the bound from  Theorem~\ref{thm: Borow} if and only if $ \rho < \col(H)$ and $k+ \rho < \col(M) + \chi_{\ell}(H)$.  We know that $k= \chi_{\ell}(M) \leq \col(M)$. So, when $\rho < \chi_{\ell}(H)$ Theorem~\ref{thm: mainresult} is certainly an improvement on Theorem~\ref{thm: Borow}.  It is also easy to see that $\rho \geq \col(H)-1$. This means that Theorem~\ref{thm: mainresult} is an improvement on Theorem~\ref{thm: Borow} if and only if $\rho = \col(H)-1$ and $\col(H) - \chi_{\ell}(H) \leq \col(M)- \chi_{\ell}(M)$ (\emph{Note:} Since we know $\chi_{\ell}(M) \leq \col(M)$ for any $M$, we can drop the second condition when $\col(H)=\chi_{\ell}(H)$).
\par
There are many examples of graphs, $M$ and $H$, that satisfy these conditions. An easy way to produce such examples is to force $H$ to be such that $\col(H)=\chi_{\ell}(H)$ and $\rho= \col(H)-1$ (examples of graphs where both of these conditions are satisfied include paths, cycles, complete graphs, and powers of paths). Then, we are free to let $M$ be any strong $k$-chromatic-choosable graph that satisfies the edge condition.  The next corollary illustrates this idea.  First, we mention a definition.  We let the \emph{kth power} of graph $G$, denoted $G^k$, be the graph with vertex set $V(G)$ where two vertices are adjacent if their distance in $G$ is at most $k$.  It is easy to see that $P_n^r$ where $1 \leq r \leq n-1$ satisfies $\chi(P_n^r) = \chi_{\ell}(P_n^r) = \col(P_n^r) = r+1$.

\begin{cor} \label{cor: withpath}
Suppose that $M$ is a strong $k$-chromatic-choosable graph that satisfies the edge condition.  Then, $M \square P_n$ is chromatic-choosable. Moreover, if $n \geq 2$, $r \leq n-1$ then $\max\{k, r+1 \} \leq \chi_{\ell}(M \square P_n^r) \leq k+r-1$
\end{cor}

\begin{proof}
Let $G_1= M \square P_n$ and $G_2=M \square P_n^r$. We note that $\chi(G_1)= \max \{ \chi(M), 2 \} =k$.  By Theorem~\ref{thm: mainresult} we have that $\chi_{\ell}(G_1) \leq k + 1 - 1 = k$.  Thus, $G_1$ is chromatic-choosable.  Similarly since $G_2$ contains a copy of both $M$ and $P_n^r$, we have that $\max\{k, r+1 \} = \max \{ \chi_{\ell}(M), \chi_{\ell}(P_n^r) \} \leq \chi_{\ell}(G_2)$. Notice that any copy of $P_n^r$ contains a Hamilton path: the underlying $P_n$ used to form $P_n^r$.  If we order the vertices based upon this Hamilton path, each vertex has at most $r$ neighbors preceding it in the ordering.  So, by Theorem~\ref{thm: mainresult} we have that $\chi_{\ell}(G_2) \leq k + r -1$.
\end{proof}

One will note that Corollary~\ref{cor: withpath} shows that the bound given by Theorem~\ref{thm: mainresult} is sharp for the graph $M \square P_n$. Suppose that $G= M \square H$ is a graph satisfying the hypotheses of Theorem~\ref{thm: mainresult}.  It is easy to show that the bound from Theorem~\ref{thm: mainresult} is sharp when $\rho=1$.  However, when $\rho > 1$ it is more difficult to determine if the bound from Theorem~\ref{thm: mainresult} is best possible.  This is because in general the obvious lower bound on $\chi_{\ell}(G)$ is $\max \{k, \chi_{\ell}(H) \}$, and the largest we can ever expect this lower bound to be is: $\max \{k, \rho+1 \}$.  We see this in Corollary~\ref{cor: withpath} with the graph $G_2$.  Specifically, since $k \geq 3$ the obvious lower bound on the list chromatic number does not tell us whether the upper bound we obtain from Theorem~\ref{thm: mainresult} is best possible.  In fact, there is a large gap between the lower and upper bound on $\chi_{\ell}(G_2)$ when both $k$ and $r$ are large.

\par

In the next section we will concentrate on developing ideas to extend the proof technique used for Lemma~\ref{lem: main} to allow for a more general second factor. This will then allow us to state a more general version of Theorem~\ref{thm: mainresult} and construct examples where our more general theorem produces sharp bounds for any $\rho \in \N$.

\section{The List Color Function and Moving Beyond Hamiltonicity}\label{listcolorfunction}

We begin by considering how we might generalize Lemma~\ref{lem: main}.  First, let us consider the case where we are taking the Cartesian product of an odd cycle (i.e. a strong 3-chromatic-choosable graph satisfying the edge condition) and a path.  Intuitively speaking, when our first factor is an odd cycle and our second factor is a path, we have a lot of freedom in the proof of Lemma~\ref{lem: main} when we color our fist copy of $M$.  In particular, we suspect that there are a lot more than just two ways to color our first copy of $M$, and we suspect that this extra freedom will allow our second factor to be more complicated.  On the other hand, when it comes to an odd cycle with arbitrarily many vertices, it is possible to find a non-constant 2-assignment, $L$, such that there are exactly 2 proper $L$-colorings for the odd cycle.  So, intuitively speaking, we may not have a lot of freedom when it comes to coloring the second copy of $M$ and onwards.

\subsection{The List Color Function} \label{sublistcolor}

In order to study the number of list colorings for a strongly chromatic-choosable graph, we need a concept that is a generalization of the chromatic polynomial. If $L$ is a list assignment for $G$, we use $P(G,L)$ to denote the number of proper $L$-colorings of $G$. The \emph{list color function} $P_\ell(G,k)$ is the minimum value of $P(G,L)$ where the minimum is taken over all possible $k$-assignments $L$ for $G$\footnote{We will allow negative integer inputs into $P_{\ell}(G,k)$, and just take $P_{\ell}(G,k)=0$ when $k < 0$.  This will make one of our results easier to state.}. The list color function is not well-understood and in general is hard to calculate. The main theme of the results from previous work, as in~\cite{KN16}, \cite{AS90}, \cite{T09}, \cite{WQ17}, is to show that $P_{\ell}(G,k)=P(G,k)$ for some special $G$ and all $k \in \N$, or for all $G$ with $k$ large enough. In this section our focus is the list color function of strongly chromatic-choosable graphs.  We begin with a general lower bound for the list color function of a strongly chromatic-choosable graph.

\begin{thm} \label{thm: stronglistcolor}
If $M$ is a strong $k$-chromatic-choosable graph, then
$$P_\ell(M,m) \geq m \max_{v \in V(M)} P_\ell(M-\{v\}, m-1) \geq m$$
whenever $m \geq k$.
\end{thm}

\begin{proof}
Suppose that $L$ is an arbitrary $m$-assignment for $M$, and suppose $v$ is an arbitrary element of $V(M)$.  We claim that for any $\alpha \in L(v)$, there is a proper $L$-coloring, $c$, for $M$ such that $c(v) = \alpha$.

\par

We construct $c$ as follows.  We begin by letting $c(v) = \alpha$ and $M' = M - \{v \}$.  Then, for each $u \in V(M')$, we let $L'(u) = L(u) - \{\alpha \}$.  Clearly, $|L'(u)| \geq m-1 \geq k-1$ for each $u \in V(M')$.  We can complete our proper $L$-coloring, $c$, for $M$ if there is a proper $L'$-coloring of $M'$.  The fact that there is a proper $L'$-coloring of $M'$ follows from Proposition~\ref{pro: obvious} Statement~(iv).

\par

Since there are $m$ colors in $L(v)$, we have that
$$P(M,L) \geq m P_\ell(M - \{v\}, m-1) \geq m.$$
Since $L$ and $v$ were arbitrary, the desired result follows.
\end{proof}

It is well known (see~\cite{R68}) that $P(C_{n},k)=(k-1)^{n}+(-1)^n(k-1)$ and $P(K_n,k) = \prod_{i=0}^{n-1} (k-i)$.  It is easy to see that for each $n,k \in \N$, $P(K_n,k)=P_{\ell}(K_n,k)$, and it was recently shown in~\cite{KN16} that for each $n,k \in \N$, $P(C_n,k) = P_{\ell}(C_n,k)$.  So, if $M$ is a strongly chromatic-choosable graph isomorphic to a complete graph or odd cycle, then $P(M,k) = P_{\ell}(M,k)$ for all $k \in \N$ and we can easily compute $P_{\ell}(M,k)$.  We now show that this also holds when $M$ is the join of an odd cycle and complete graph.

Using a classic result on the chromatic polynomial of the join of two graphs (see for example~\cite{B94}), it is easy to see that for any graph $G$ and $n \in \N$, $P(G \vee K_n , k) = P(G, k-n)P(K_n, k)$.  We now can prove some bounds on the list color function of an arbitrary graph joined with a complete graph.

\begin{pro} \label{pro: listfunctionjoin}
For any graph $G$ and $n, k \in \N$,
$$P_{\ell}(G, k-n) P(K_n, k) \leq P_{\ell}(G \vee K_n , k) \leq P(G \vee K_n , k) = P(G, k-n)P(K_n, k) .$$
In particular, when $P_{\ell}(G, m) = P(G, m)$ for all $m \in \N$, $P_{\ell}(G \vee K_n , k) = P(G \vee K_n , k)$ for all $k \in \N$.
\end{pro}

\begin{proof}
The second inequality is trivial.  So, we just prove the first inequality.  Suppose $H = G \vee K_n$, $G_1$ is the copy of $K_n$ used to form $H$, and $G_2$ is the copy of $G$ used to form $H$.  The result is trivial when $k < n + \chi_{\ell}(G)$.  So, assume that $k \geq n + \chi_{\ell}(G)$ and $L$ is an arbitrary $k$-assignment for $H$.  Suppose we find a proper $L$-coloring of $H$ by first coloring $G_1$ then coloring $G_2$.  Notice that there are at least $P(K_n, k)$ possible proper $L$-colorings of $G_1$.  After we color $G_1$, there are at least $k-n$ possible color choices in $L(v)$ that may be used on each $v \in V(G_2)$.  Thus, there are at least $P_{\ell}(G, k-n)$ possible proper $L$-colorings of $G_2$.  So, $P_{\ell}(G, k-n) P(K_n, k) \leq P(H,L)$.  The result follows since $L$ was arbitrary.
\end{proof}


Combining results mentioned thus far, we have the following corollary.

\begin{cor} \label{lem: listfunction2}
For $n, l, k \in \N$,
$$P_{\ell}(C_{2l+1} \vee K_n,k)=P(C_{2l+1} \vee K_n,k)=\left[ (k-n-1)^{2l+1}-(k-n-1) \right] \prod_{i=0}^{n-1} (k-i).$$
\end{cor}

\subsection{Chromatic-Choosability with Stars}

We will now prove a result in the spirit of Lemma~\ref{lem: main}.  The idea driving the following lemma is that if we take the Cartesian product of a strongly chromatic-choosable graph and a copy of $K_{1,s}$ (i.e. a star graph), we should be able to prove the graph is chromatic-choosable for certain $s>1$.  For the lemma it is useful to note that $P_{\ell}(G,k) \geq 2$ when $G$ is a strong $k$-chromatic-choosable graph by Theorem~\ref{thm: stronglistcolor} (this is under the usual assumption that $k \geq 2$).

\begin{lem} \label{lem: star}
Let $M$ be a strong $k$-chromatic-choosable graph with $V(M)= \{u_1, \ldots, u_{n}\}$.  Suppose $s < P_{\ell}(M,k)$, and $B=K_{1,s}$ with partite sets $X=\{v_0 \}$ and $Y=\{v_1, \ldots, v_s \}$.  Let $G= M \square B$, and for $0 \leq i \leq s$ let $V_i$ be the subset of $V(G)$ that consists of all the vertices with second coordinate $v_i$.   Suppose that $L$ is an arbitrary $k$-assignment for $G$.  Then, there exists a proper coloring, $c$, for $G[V_0]$ such that $c(v) \in L(v)$ for each $v \in V_0$ and if $L'$ is the list assignment for the vertices in $V(G)-V_0$ given by $L'(u_j,v_i)=L(u_j,v_i)- \{c(u_j,v_0) \}$ for each $(u_j,v_i) \in V(G)-V_0$, then we obtain the following results depending on whether $M$ satisfies the edge condition.
\\
(i)  In the case $M$ satisfies the edge condition, for each $1 \leq i \leq s$ there exists two distinct proper colorings, $c_{i,1}$ and $c_{i,2}$, for $G[V_i]$ such that $c_{i,t}(v) \in L'(v)$ for each $v \in V_i$ and $t=1,2$.
\\
(ii)	In the case $M$ does not satisfy the edge condition, for each $1 \leq i \leq s$ there exists a proper coloring, $c_i$, for $G[V_i]$ such that $c_i(v) \in L'(v)$ for each $v \in V_i$.
\end{lem}

\begin{proof}
We first note that by definition, there are at least $P_{\ell}(M,k)$ proper colorings for $G[V_0]$ that assign a color in $L(v)$ to $v$ for each $v \in V_0$. Let $\mathcal{C}$ be the set of all these colorings.  For each $i \geq 1$, we refer to $c \in \mathcal{C}$ as a \emph{bad coloring for $G[V_i]$} if the list assignment $L''$ for $G[V_i]$ given by $L''(u_j,v_i)=L(u_j,v_i)- \{ c(u_j,v_0) \}$ for each $(u_j,v_i) \in V_i$ is a constant $(k-1)$-assignment for $G[V_i]$.  By the argument in the inductive step of the proof of Lemma~\ref{lem: main} we know that for each $i \geq 1$ there is at most one bad coloring for $G[V_i]$ in $\mathcal{C}$.  Now, we know that:
$$s < P_{\ell}(M,k) \leq |\mathcal{C}|.$$
So, we may conclude that there exists a $c \in \mathcal{C}$ such that $c$ is not a bad coloring for $G[V_i]$ for any $i$ satisfying $1 \leq i \leq s$.  Now, let $L'$ be the list assignment for the vertices in $V(G)-V_0$ given by $L'(u_j,v_i)=L(u_j,v_i)- \{c(u_j,v_0) \}$ for each $(u_j,v_i) \in V(G)-V_0$.  It is easy to see that $|L'(v)| \geq k-1$ for each $v \in V(G)-V_0$.  Moreover, when $L'$ is restricted to $G[V_i]$ for each $i$, we get a list assignment for $G[V_i]$ that is not a constant $(k-1)$-assignment for $G[V_i]$.  So, by Lemma~\ref{lem: notunique}, when $M$ satisfies the edge condition, there must be at least $2$ distinct proper colorings for $G[V_i]$ that assign a color in $L'(v)$ to $v$ for each $v \in V_i$.  By Proposition~\ref{pro: obvious} Statement (ii), when $M$ does not satisfy the edge condition, there is at least $1$ proper coloring for $G[V_i]$ that assigns a color in $L'(v)$ to $v$ for each $v \in V_i$.
\end{proof}

We are now in a position to give the list chromatic number of the Cartesian product of a strongly chromatic-choosable graph and a star.  One should note that, intuitively speaking, a star with many leaves is far from containing a Hamilton path.  So, we see that the result of Lemma~\ref{lem: star} is a start in generalizing Theorem~\ref{thm: mainresult}.

\begin{thm} \label{thm: star}
Let $M$ be a strong $k$-chromatic-choosable graph.  Then,
\[
\chi_{\ell}(M \square K_{1,s}) =
\begin{cases}
k & \text{if } s < P_{\ell}(M,k)\\
k+1 & \text{if } s \geq P_{\ell}(M,k).
\end{cases}
\]
\end{thm}

\begin{proof}
The fact that $\chi_{\ell}(M \square K_{1,s})=k$ when $s < P_{\ell}(M,k)$ follows from Lemma~\ref{lem: star} and the fact that $\chi(M)=k$.  So, suppose that $s \geq P_{\ell}(M,k)$.  As in the proof of Lemma~\ref{lem: star}, let $V(M)= \{u_1, u_2, \ldots, u_{n}\}$.  Suppose $B=K_{1,s}$ with partite sets $X=\{v_0 \}$ and $Y=\{v_1, \ldots, v_s \}$.  Let $G= M \square B$, and for $0 \leq i \leq s$ let $V_i$ be the subset of $V(G)$ that consists of all the vertices with second coordinate $v_i$.  We need to show that $\chi_{\ell}(G) = k+1$.  Since $\col(B)=2$, Theorem~\ref{thm: Borow} implies that $\chi_{\ell}(G) \leq k+2-1 = k+1$.  So, we need to construct a $k$-assignment, $L$, for $G$ such that there is no proper $L$-coloring for $G$.

\par

In order to construct $L$, we begin by assigning to the vertices in $V_0$ lists of size $k$ such that there are precisely $P_{\ell}(M,k)$ proper list colorings of $G[V_0]$.  We let $t = P_{\ell}(M,k)$ and we let $\mathcal{C} = \{c_1, c_2, \ldots, c_t \}$ be the set of proper list colorings of $G[V_0]$.  Now, let $B$ be a set of $(k-1)$ elements each of which is not in $\cup_{i=1} ^ n L(u_i,v_0)$.  For, $1 \leq j \leq t$ and $1 \leq i \leq n$, let
$$L(u_i,v_j) = B \cup \{c_j(u_i,v_0) \}.$$
If $s > t$ complete the list assignment, $L$, by arbitrarily assigning $k$-element lists to any vertices in $V(G)$ that have second coordinate $v_j$ where $j > t$.

\par

Now, for the sake of contradiction assume that $c$ is a proper $L$-coloring for $G$.  It must be that there is a $c_r \in \mathcal{C}$ such that $c(v) = c_r(v)$ for each $v \in V_0$ (since $c$ must properly color $G[V_0]$).  This means that $c$ restricted to $V_r$ is a proper $L'$-coloring of $G[V_r]$ where $L'$ is the list assignment given by:
$$L'(u_i,v_r)=L(u_i,v_r)-\{c_r(u_i,v_0) \} = B$$
for each $(u_i,v_r) \in V_r$.  So, $L'$ is a constant $(k-1)$-assignment for $G[V_r]$.  We now have a contradiction since $G[V_r]$ is a copy of $M$, and by definition, it is impossible to obtain a proper coloring of $M$ from a constant $(k-1)$-assignment.
\end{proof}

Lemma 3 in~\cite{BJ06} implies (among other things) that $\chi_{\ell}(C_{2l+1} \square K_{1,s})=4$ when $s \geq 3^{2l+1}$ and $\chi_{\ell}(K_n \square K_{1,s})=n+1$ when $s \geq n^n$.  The following corollary, which immediately follows from the results mentioned in Subsection~\ref{sublistcolor} and Theorem~\ref{thm: star}, improves upon these results, and completely solves the problem of finding the list chromatic number of the Cartesian product of an odd cycle and star, the list chromatic number of the Cartesian product of a complete graph and star, and the list chromatic number of the Cartesian product of a star and the join of a complete graph with an odd cycle.

\begin{cor} \label{cor: oddstar}
For any $n , l \in \N$, we have that:
\\
\[
\chi_{\ell}(C_{2l+1} \square K_{1,s}) =
\begin{cases}
3 & \text{if } s < 2^{2l+1}-2\\
4 & \text{if } s \geq 2^{2l+1}-2.
\end{cases}
\]

\[
\chi_{\ell}(K_n \square K_{1,s}) =
\begin{cases}
n & \text{if } s < n!\\
n+1 & \text{if } s \geq n!.
\end{cases}
\]

\[
\chi_{\ell}((K_n \vee C_{2l+1}) \square K_{1,s}) =
\begin{cases}
n+3 & \text{if } s < \frac{1}{3} (n+3)! (4^l-1)\\
n+4 & \text{if } s \geq \frac{1}{3} (n+3)! (4^l-1).
\end{cases}
\]

\end{cor}

For any graph $G$ we say that $H$ is a \emph{subdivision of $G$} if $H$ is a graph obtained from $G$ by replacing the edges of $G$ with pairwise internally disjoint paths.  By using the proof idea of Lemma~\ref{lem: main} and the result of Lemma~\ref{lem: star} Statement (i), we can obtain the following corollary.

\begin{cor} \label{cor: subdivision} Let $M$ be a strong $k$-chromatic-choosable graph that satisfies the edge condition, and $B'$ be a subdivision of the star $K_{1,s}$ with $s < P_{\ell}(M,k)$.  Then, $\chi_{\ell}(M \square B')=k$, that is, $M \square B'$ is chromatic-choosable.
\end{cor}

\subsection{Generalizing Theorem~\ref{thm: mainresult} with sharpness for $\rho>1$}

With Corollary~\ref{cor: subdivision} in mind, we are ready to observe a generalization of Theorem~\ref{thm: mainresult} that allows for more general second factors.  The proof relies on combining the proof ideas of Lemmas~\ref{lem: main} and~\ref{lem: star}.  Specifically, we will introduce the concept of \emph{$(M, \rho)$-Cartesian accommodating}, and we will prove the following theorem.

\begin{thm} \label{thm: general}  Suppose that $M$ is a strong $k$-chromatic-choosable graph that satisfies the edge condition, and suppose that $H$ is a $(M, \rho)$-Cartesian accommodating graph.  Then,
$\chi_{\ell}(M \square H) \leq k + \rho -1.$
\end{thm}

We will see that graphs that contain a Hamilton path, $w_1, w_2, \ldots, w_m$, such that $\rho \geq 1$ and $w_i$ has at most $\rho$ neighbors among $w_1, \ldots, w_{i-1}$ are $(M,\rho)$-Cartesian accommodating along with many other classes of graphs.  So, Theorem~\ref{thm: general} truly generalizes Theorem~\ref{thm: mainresult}.  We will now introduce the concept of $(M,\rho)$-Cartesian accommodating.

\par

Let $M$ be a given strong $k$-chromatic-choosable graph. Suppose $H$ is a graph such that $V(H)$ can be partitioned into independent sets: $I_1, I_2, \ldots, I_s$.  Let $\eta : V(H) \rightarrow \Z$ be the function defined so that for each $v \in I_{\lambda}$, $1 \leq \lambda \leq s$, $\eta(v)$ is the number of neighbors $v$ has in $ \cup_{i=1}^{\lambda-1} I_i$ (we take this union to be the empty set when $\lambda=1$).  Let $\rho \geq \max \{1,\max_{v \in V(H)} \eta(v)\}$, and $F = \{v \in V(H) | \eta(v) = \rho \}$.  Suppose that $H$ satisfies:
\begin{align*}
&(1)  \text{ For each $\lambda \geq 2$, every $v \in I_{\lambda} \cap F$ is adjacent to at least one vertex in $I_{\lambda-1}$, and}\\
&(2) \text{ For each $\lambda$ satisfying $1 \leq \lambda \leq s-1$, each $v \in I_{\lambda}-F$, $v$ is adjacent to less} \\
& \; \; \; \; \; \text{than $P_{\ell}(G, (k+ \rho -1)- \eta(v))$ vertices in $I_{\lambda+1} \cap F$.  Also for each $v \in I_{\lambda} \cap F$, $v$ is adjacent} \\
& \; \;\ \; \; \text{to at most 1 vertex  in $I_{\lambda+1} \cap F$.}
\end{align*}
We call $H$ \emph{$(M, \rho)$-Cartesian accommodating} when it satisfies these conditions.  Note that we can ensure condition (1) is satisfied by placing a vertex in the independent set $I_\lambda$ with smallest index possible when there is a choice (though there is no guarantee that after this is done condition (2) will be satisfied).  As with the proof of Theorem~\ref{thm: mainresult}, we will prove a lemma which will immediately imply Theorem~\ref{thm: general}.

\par

Suppose that $M$ is a strong $k$-chromatic-choosable graph that satisfies the edge condition with $V(M) = \{v_1, \ldots, v_n \}$.  Suppose $H$ is a $(M, \rho)$-Cartesian accommodating graph (we use the same notation as in the definition).  Let $G = M \square H$.  For each $u \in V(H)$, let $V_{u}$ represent the vertices in $V(G)$ with $u$ as the second coordinate.  Now, the following lemma immediately implies Theorem~\ref{thm: general}.

\begin{lem} \label{lem: general}
Let $L$ be an arbitrary $(k + \rho -1)$-assignment for $G$.  There exists a proper $L$-coloring, $c$, of $M \square (H - I_s)$ that satisfies the following conditions.  For each $v \in \cup_{u \in I_s} V_u$, let $L'(v)$ be the list obtained from $L(v)$ by deleting any colors used by $c$ on vertices adjacent to $v$ in $V(M \square (H - I_s))$.  For each $u \in I_s \cap F$ there are at least 2 proper $L'$-colorings of $G[V_u]$, and for each $u \in I_s - F$ there are at least $P_{\ell}(M, k+ \rho - 1 - \eta(u))$ proper $L'$-colorings of $G[V_u]$.
\end{lem}

\begin{proof}
The proof is by induction on $s$.  For the base case suppose that $s=1$.  Note $M \square (H - I_s)$ is empty, $\eta(u) = 0$ for each $u \in V(H)$, and there are at least $P_{\ell}(M, k + \rho - 1)$ ways to properly color $G[V_u]$ for each $u \in V(H)$ (since $\rho \geq 1$).  Thus, the base case is complete.

\par

For the induction step assume that $s \geq 2$.  Also assume there exists a proper $L$-coloring, $c$, of $M \square (H - (I_s \cup I_{s-1}))$ that satisfies the following conditions.  For each $v \in \cup_{u \in (I_{s-1} \cup I_s)} V_u$, let $L'(v)$ be the list obtained from $L(v)$ by deleting any colors used by $c$ on vertices adjacent to $v$ in $V(M \square (H - (I_s \cup I_{s-1})))$.  For each $u \in I_{s-1} \cap F$ there are at least 2 proper $L'$-colorings of $G[V_u]$, and for each $u \in I_{s-1} - F$ there are at least $P_{\ell}(M, k+ \rho - 1 - \eta(u))$ proper $L'$-colorings of $G[V_u]$.  We note that since $w \in I_{s} \cap F$ is adjacent to at least one vertex in $I_{s-1}$, each $v \in \cup_{u \in I_s} V_u$ satisfies:
$$|L'(v)| \geq k.$$

\par

Now, suppose that $|I_{s-1}|=a$ and $I_{s-1} = \{w_1, \ldots, w_a \}$.  For each $u \in I_{s-1}$ we will pick a proper $L'$-coloring of $G[V_u]$ that will allow us to prove the desired.  We will do this ``in order" by picking a proper $L'$-coloring of $G[V_{w_1}]$, followed by $G[V_{w_2}]$, $\ldots$, and finally $G[V_{w_a}]$.

\par

We first describe how we pick a proper $L'$-coloring of $G[V_{w_1}]$.  In the case that $w_1 \in I_{s-1} \cap F$, we know that, in $H$, $w_1$ is adjacent to at most 1 vertex in $I_s \cap F$, and there are at least 2 proper $L'$-colorings for $G[V_{w_1}]$.  Suppose $A$ is the set of vertices in $I_s \cap F$ adjacent to $w_1$ in $H$.  So, we choose a proper coloring for $G[V_{w_1}]$ that will not lead to a constant $(k-1)$-assignment for the copy of $M$ corresponding to $G[V_{u}]$.  In the case where $w_1$ is adjacent to no vertices in $ I_s \cap F$ we arbitrarily pick a proper $L'$-coloring for $G[V_{w_1}]$.
\par
Now, consider the case that $w_1 \in I_{s-1} - F$.  In this case we know that in $H$, $w_1$ is adjacent to less than $P_{\ell}(M, k+ \rho - 1 - \eta(w_1))$ vertices in $I_s \cap F$, and there are at least $P_{\ell}(M, k+ \rho - 1 - \eta(w_1))$ proper $L'$-colorings for $G[V_{w_1}]$. Suppose $w_1$ is adjacent to all the vertices in $A \subseteq I_s \cap F$.  Then, we pick a proper $L'$-coloring for $G[V_{w_1}]$ that does not lead to a constant $(k-1)$-assignment for any copy of $M$ of the form $G[V_u]$ where $u \in A$.
\par
After choosing a proper $L'$-coloring of $G[V_{w_1}]$, let $L^{(2)}$ be the list assignment obtained in the following way.  For each $v \in \cup_{u \in I_s} V_u$, let $L^{(2)}(v)$ be the list obtained from $L'(v)$ by deleting any colors used by the proper $L'$-coloring chosen for $G[V_{w_1}]$ on any vertices in $V_{w_1}$ that are neighbors of $v$ in $G$.  We continue coloring the copies of $M$ by following the outline for coloring $G[V_{w_1}]$ described above.  At each stage, for each $v \in \cup_{u \in I_s} V_u$, we let $L^{(t)}(v)$ be the list obtained from $L^{(t-1)}(v)$ by deleting any colors used by the proper $L'$-coloring chosen for $G[V_{w_{t-1}}]$ on any vertices in $V_{w_{t-1}}$ that are neighbors of $v$.  Note that for any $t \geq 2$, if $|L^{(t)}(v)| = k-1$, then $v$ has no neighbors in the yet to be colored copies of $M$ since in this case $\rho$ colors must have been deleted from $L(v)$ to get $L^{(t)}(v)$, and $v$ has at most $\rho$ neighbors with second coordinate in $V(H)-I_s$.  Also, if $L^{(t)}$ restricted to $G[V_u]$ for some $u \in I_s$ is a $(k-1)$-assignment, it must be a non-constant $(k-1)$-assignment.

\par

After we have colored all the copies of $M$ we are left with a list assignment, $L^{(a+1)}$, for each $v \in \cup_{u \in I_s} V_u$.  We notice that for each $u \in I_s \cap F$, we have that for each $v \in V(G[V_u])$,
$$|L^{(a+1)}(v)| \geq k-1$$
and $L^{(a+1)}$ restricted to $G[V_u]$ is not a constant $(k-1)$-assignment.  Thus, there are at least two proper $L^{(a+1)}$- colorings of $G[V_u]$ by Lemma~\ref{lem: notunique}.  Also, for each $u \in I_s - F$ we have that for each $v \in V(G[V_u])$,
$$|L^{(a+1)}(v)| \geq k+ \rho - 1 - \eta(u) \geq k.$$
 Thus, there are at least $P_{\ell}(M, k+ \rho - 1 - \eta(u))$ proper $L^{(a+1)}$- colorings of $G[V_u]$.  Hence the induction is complete.

\end{proof}

It is worth mentioning that if $F \subseteq I_s$ we do not need $M$ to satisfy the edge condition in order to obtain the upper bound on the list chromatic number of Theorem~\ref{thm: general}.

\par

We now show that Theorem~\ref{thm: general} generalizes Theorem~\ref{thm: mainresult}.  To see this suppose that $H$ is a graph that contains a Hamilton path, $w_1, w_2, \ldots, w_m$, such that $\rho \geq 1$ and $w_i$ has at most $\rho$ neighbors among $w_1, \ldots, w_{i-1}$.  If for $1 \leq \lambda \leq m$, we let $I_{\lambda} = \{w_{\lambda}\}$, we see that for any strong $k$-chromatic-choosable graph that satisfies the edge condition, $M$, $H$ is $(M, \rho)$-Cartesian accomodating.

Not only does Theorem~\ref{thm: general} generalize Theorem~\ref{thm: mainresult}, but we can also show that there exist examples of $M$ and $H$ where Theorem~\ref{thm: general} gives a tight bound for any $\rho \in \N$.  Specifically, suppose that $M$ is a strong $k$-chromatic-choosable graph that satisfies the edge condition.  Let $B'$ be some subdivision of $K_{1, P_{\ell}(M,k)-1}$.  For each $t \in \N$ we define the $S_{M,B',t}$ graph inductively.  Let $S_{M,B',1}=B'$.  Then, for $t \geq 2$ we construct $S_{M,B',t}$ as follows.  Take $P_{\ell}(M,k+t-2)$ disjoint copies of $S_{M,B',t-1}$ and join a single vertex to these copies.  The reason we use $P_{\ell}(M,k+t-2)$ copies will become clear in a moment.
\par
Several properties of $S_{M,B',t}$ are immediate.  In particular for each $t \in \N$, $\chi(S_{M,B',t})=t+1$ and $\col(S_{M,B',t})=t+1$.  Moreover, $S_{M,B',t}$ is $(M, t)$-Cartesian accommodating.  So, by Theorem~\ref{thm: general}, $\chi_{\ell}(M \square S_{M,B,t}) \leq k + t -1$.  We now show that this upper bound is best possible.

\begin{pro} \label{pro: generalapplication}  Let $M$ be a strong $k$-chromatic-choosable graph that satisfies the edge condition.  Then, for any $t \in \N$, $\chi_{\ell}(M \square S_{M,B',t}) = k + t -1$.
\end{pro}

\begin{proof}
For each $t \in \N$ we need only show that there exists a bad $(k+t-2)$-assignment for $M \square S_{M,B',t}$.  We will show what is required by induction on $t$.  The statement when $t=1$ is obvious since we could simply associate a constant $(k-1)$-assignment with each copy of $M$ in $M \square S_{M,B',1}$.

\par

Now, suppose that $t \geq 2$ and the desired statement holds for all natural numbers less than $t$.  Let $H_1, H_2, \ldots, H_{P_{\ell}(M,k+t-2)}$ represent the $P_{\ell}(M,k+t-2)$ disjoint copies of $S_{M,B',t-1}$ used to form $S_{M,B',t}$.  We suppose that each of these copies have $l$ vertices.  Let $v$ represent the single vertex joined to these copies to form $S_{M,B',t}$. Let $v_{1,i}, \ldots, v_{l,i}$ represent the vertices in $V(H_i)$ for $1 \leq i \leq P_{\ell}(M,k+2-t)$.  Let $u_1, u_2, \ldots, u_n$ be the vertices in $V(M)$.  By the inductive hypothesis we know that there is a bad $(k+t-3)$-assignment, $L$, for each of the copies of $M \square S_{M,B',t-1}$ in $M \square S_{M,B',t}$. Let $L'$ be a $(k+t-2)$-assignment for $M$ such that $M$ has precisely $P_{\ell}(M,k+t-2)$ proper $L'$-colorings.  We choose the colors in $L'$ such that none of the colors are in any of the lists associated with $L$.  Suppose that the proper $L'$-colorings of $M$ are $c_1, c_2, \ldots, c_{P_{\ell}(M,k+t-2)}$.

\par

We form a $(k+t-2)$-assignment, $L''$, for $M \square S_{M,B',t}$ as follows.  First, we define $L''$ for the copy of $M$ corresponding to $v$. For each $1 \leq r \leq n$ let
$$L''(u_r,v)=L'(u_r).$$
Now, as $i$ varies from 1 to $P_{\ell}(M,k+t-2)$ we define $L''$ for the vertices in each copy of $M \square H_i$ in $M \square S_{M,B',t}$.   Specifically, for each $1 \leq i \leq P_{\ell}(M,k+t-2)$, $1 \leq j \leq l$, and $1 \leq r \leq n$ let
$$L''(u_r, v_{j,i})=L(u_r, v_{j,i}) \cup \{c_i(u_r) \}.$$
Finally, suppose there is a proper $L''$-coloring of $M \square S_{M,B',t}$.  This proper coloring contains a proper $L'$-coloring, assume it is $c_m$, of the copy of $M$ in $M \square S_{M,B',t}$ corresponding to $v$.  Now, in order for there to be a proper $L''$-coloring of $M \square S_{M,B',t}$, there must be a proper $L$-coloring of the $M \square H_m$ in $M \square S_{M,B',t}$ since $L''(u_r, v_{j,m})- \{c_m(u_r) \}=L(u_r, v_{j,m})$
for each $1 \leq j \leq l$ and $1 \leq r \leq n$.  However, this is impossible.  Thus, $L''$ is a bad $(k+t-2)$-assignment for $M \square S_{M,B',t}$ and our proof is complete.
\end{proof}

We will now conclude this section with an illustrative example.  We know that $C_3$ is strong 3-chromatic-choosable graph that satisfies the edge condition.  We also know that $P_{\ell}(C_3,k) = P(C_3,k) = k(k-1)(k-2).$ Suppose that $H_1 = K_{1,5}$.  Since $P_{\ell}(C_3,3)=6$, it is immediately clear that $H_1$ is $(C_3,1)$-Cartesian accommodating.  Now, for $m \geq 2$, suppose that $H_m$ is graph obtained by taking $P_{\ell}(C_3,m+1)$ disjoint copies of the graph $H_{m-1}$ and join a single vertex to these copies.  So, $H_2 = K_1 \vee 6H_1$, $H_3 = K_1 \vee 24H_2$, $H_4 = K_1 \vee 60H_3$, etc.  We have that $H_m$ is $(C_3,m)$-Cartesian accommodating.  The following facts are also clear:
$ \chi(H_m)=m+1 \; \; \text{and} \; \; \col(H_m) = m+1.$  So, Theorem~\ref{thm: Borow} implies that $\chi_{\ell}(C_3 \square H_m) \leq 3 + m+ 1 -1 = m+3.$  However, Theorem~\ref{thm: general} implies that $\chi_{\ell}(C_3 \square H_m) \leq 3+m-1= m+2,$ and Proposition~\ref{pro: generalapplication} implies that $\chi_{\ell}(C_3 \square H_m)= m+2.$

\appendix

\section{Appendix}

We begin with a proof of Proposition~\ref{pro: addaddtoodd}.

\begin{pro*} [\textbf{\ref{pro: addaddtoodd}}] For $k \in \N$ and $m \in \{1,2,3\}$ we construct $G_{l,m,k}$ inductively as follows.  For $k=1$, $G_{l,m,k}$ is the graph constructed in the statement of Proposition~\ref{pro: addtoodd}.  For $k \geq 2$ we construct $G_{l,m,k}$ from $G_{l,m,k-1}$ as follows.  We add vertices $u_k$ and $s_k$ to $G_{l,m,k-1}$ and we add edges so that $u_k$ is adjacent to $\{u_j | 1 \leq j \leq k-1 \} \cup \{v_j | 1 \leq j \leq 2+m \}$ and so that $s_k$ is adjacent to $\{s_j | 1 \leq j \leq k-1 \} \cup (V(C)-\{v_1,v_2 \})$.  Then, $G_{l,m,k}$ is strong $(3+k)$-chromatic-choosable.  Moreover, $G_{l,m,k}$ is not strong $(3+k)$-critical when $m=2,3$.
\end{pro*}

\begin{proof}  The proof is by induction on $k$.  Notice that the base case is the result of Proposition~\ref{pro: addtoodd}.  Now, suppose that the desired result holds for all natural numbers less than $k$ where $k \geq 2$.  So, we know that $G_{l,m,k-1}$ is strong $(2+k)$-chromatic-choosable.  Note that $\{u_j | 1 \leq j \leq k \}$ is a clique in $G_{l,m,k}$ adjacent to all the vertices in the path $P_1$ given by $v_1, v_2, v_3$.  Similarly, $\{s_j | 1 \leq j \leq k \}$ is a clique in $G_{l,m,k}$ adjacent to all the vertices in the path $P_2$ given by $v_3, v_4, \ldots, v_{2l+1}$.  Since we know that for any proper coloring of $C$, $P_1$ or $P_2$ must be colored with at least 3 colors, we have that $\chi(G_{l,m,k}) > k+2$.  Now, let $A= \{u_j | 1 \leq j \leq k-1 \} \cup \{v_j | 1 \leq j \leq 2+m \}$ and $B=\{s_j | 1 \leq j \leq k-1 \} \cup (V(C)-\{v_1,v_2 \})$.  We note that $A \cup B = V(G_{l,m,k-1})$, $|A \cap B| = m \leq 3$, $|A| > |A \cap B|$, and $|B| > |A \cap B|$.  Thus, Lemma~\ref{lem: construct} immediately implies that $G_{l,m,k}$ is strong $(3+k)$-chromatic-choosable.  It is also easy to see that $G_{l,2,k}$ and $G_{l,3,k}$ are not $(3+k)$-critical since they contain a copy of $G_{l,1,k}$ as a proper subgraph.
\end{proof}

Now we prove that Lemma~\ref{lem: construct} can be extended in the case where our starting graph is an odd cycle. The proof of this extension relies on the following proposition.

\begin{pro} \label{pro: fodd}  Let $G$ be $C_{2l+1}$ with its vertices in cyclic order as: $v_1, v_2, \ldots, v_{2l+1}$.  Let $f: V(G) \rightarrow \N$ be a function such that there exists $i$ and $j$ with: $f(v_i)=1$, $f(v_j)=3$, and $f(v_t)=2$ whenever $t \neq i$ and $t \neq j$.  Then, $G$ is $f$-choosable.
\end{pro}

\begin{proof}  Without loss of generality, suppose that $i=1$.  Suppose that $L$ is an arbitrary list assignment for $G$ such that $|L(v)|=f(v)$ for each $v \in V(G)$.  To prove the desired result we need only show that $G$ is $L$-colorable.  Suppose we order the vertices of $G$ as follows:

$$v_1, v_2, \ldots v_{j-1}, v_{2l+1}, v_{2l}, v_{2l-1}, \ldots, v_j.$$

\noindent One should note that it is possible $j-1=1$ or $j=2l+1$.  We notice that in the above ordering $v_1$ has no neighbors preceding it, $v_j$ has 2 neighbors preceding it, and each other vertex has one neighbor preceding it.  Thus, we can use the vertex ordering to greedily find a proper $L$-coloring for $G$.
\end{proof}

\begin{lem*} [\textbf{\ref{lem: constructodd}}] Let $G$ be an odd cycle $C_{2l+1}$. Suppose we can find sets $A, B \subseteq V(G)$ such that $A \cup B = V(G)$. Let $C=A \cap B$ and suppose  $ 0 < |C| \leq 8$, and $|A|, \;|B| > |C|$.  Form $G'$ by adding vertices $u$ and $s$ to $G$, and add edges so that $u$ is adjacent to every vertex in $A$ and $s$ is adjacent to every vertex in $B$.  If $\chi(G') > 3$, then $G'$ is strong $4$-chromatic-choosable.
\end{lem*}

\begin{proof}  Note that when $|C| \leq 4$ the desired result immediately follows from Lemma~\ref{lem: construct}.  So, we assume that $5 \leq |C| \leq 8$. We let $C=\{v_1, \ldots, v_m \}$.  Suppose that $L$ is an arbitrary non-constant $3$-assignment for $G'$.  In order to show that $G'$ is strong $4$-chromatic-choosable, we must show $G'$ is $L$-colorable.  We know that $G=G'-\{u, s \}$.  We may assume that $L(u) \cap L(s) = \emptyset$, since the idea of the proof of case (i) in the proof of Lemma~\ref{lem: construct} proves the desired when $L(u) \cap L(s) \neq \emptyset$.  As in the proof of Lemma~\ref{lem: construct}, let $I=\bigcap_{i=1}^m L(v_i)$.  We can assume that $|I| < 2$ since following the idea presented for sub-case (a) in the proof of Lemma~\ref{lem: construct} yields the desired when $|I| \geq 2$.

\par

Suppose that $L(u)= \{c_1, c_2, c_3 \}$ and $L(v)= \{c_4, c_5, c_6 \}$.  The 9 possible ways to color $u$ and $s$ are represented by the 9 color pairs in the set $P=\{(c_i,c_j)|i \in \{1,2,3\}, j \in \{4,5,6\} \}$.  Since $m \leq 8$ and each list contains at most 2 color pairs, there must be a color pair in $P$ that is contained in at most one of the lists $L(v_1), \ldots, L(v_m)$.  Without loss of generality suppose $(c_1, c_4)$ is contained in at most one of these lists.  Note that if any of the color pairs in $P$ are contained in none of the lists $L(v_1), \ldots, L(v_m)$, we can find a proper $L$-coloring for $G'$ by following the idea of sub-case (b) in the proof of Lemma~\ref{lem: construct}.  So, we may assume that $(c_1, c_4)$ appears in exactly one of the lists: $L(v_1), \ldots, L(v_m)$.  We now consider two cases.  Specifically, we consider the cases: (1) there is some $j$ ($1 \leq j \leq m$) such that $L(v_j)$ contains neither $c_1$ nor $c_4$ and (2) Each list: $L(v_1), \ldots, L(v_m)$ contains $c_1$ or $c_4$.

\par

For (1) we color $u$ with $c_1$ and $s$ with $c_4$, and for each $v \in V(G)$ we let
\[ L'(v) = \begin{cases}
      L(v) - \{c_1\} & \textrm{ if $v \in A-C$} \\
      L(v) - \{c_4\} & \textrm{ if $v \in B-C$} \\
			L(v) - \{c_1,c_4\} & \textrm{if $v \in C$.}\\
   \end{cases} \]
We note that there is exactly one vertex in $G$ to which $L'$ assigns one color, and all other vertices in $G$ have at least two colors assigned to them by $L'$.  Moreover, $|L'(v_j)| = 3$.  So, Proposition~\ref{pro: fodd} immediately implies that we can complete a proper $L$-coloring of $G'$.

\par

For (2) we let $P'= \{(c_i,c_j)|i \in \{2,3\}, j \in \{5,6\} \}$.  We consider the lists $L(v_1), \ldots, L(v_m)$.  Of these lists, the list containing $(c_1,c_4)$ must contain no color pairs in $P'$, and since all these lists contain $c_1$ or $c_4$, the remaining lists contain at most one color pair in $P'$.  Now, suppose that every color pair in $P'$ occurs in at least two of the lists: $L(v_1), \ldots, L(v_m)$.  Since each $L(v_2), \ldots, L(v_m)$ can accommodate at most one pair in $P'$ we need $m \geq 9$.  Thus, there must be a color pair in $P'$ that is contained in at most one of these lists.  Without loss of generality, suppose that $(c_2, c_5)$ is such a color pair.  We may assume each list: $L(v_1), \ldots, L(v_m)$ contains $c_2$ or $c_5$ since otherwise we may obtain a proper $L$-coloring for $G'$ by proceeding as we did in case (1).  Since each list contains $c_2$ or $c_5$ and we already know each list contains $c_1$ or $c_4$, we know that none of the lists contain the color pair $(c_3, c_6)$.  This completes case (2) and we are finished.
\end{proof}

\end{document}